\documentclass{amsart}
\usepackage{amsmath}

\usepackage{amssymb}
\usepackage{amsthm}
\usepackage{amscd,amsbsy}
\usepackage{xcolor}
\usepackage{tikz-cd}
\usetikzlibrary{arrows,matrix}
\usepackage{float}
\usepackage{mathtools}
\usepackage{enumitem}

\newtheorem{theorem}{Theorem}[section]
\newtheorem{lemma}[theorem]{Lemma}
\newtheorem{proposition}[theorem]{Proposition}
\newtheorem{corollary}[theorem]{Corollary}
\newtheorem{claim}[theorem]{Claim}

\theoremstyle{definition}
\newtheorem{definition}[theorem]{Definition}

\newtheorem{example}[theorem]{Example}
\newtheorem{examples}[theorem]{Examples}
\newtheorem{definitions and remarks}[theorem]{Definitions and Remarks}
\newtheorem{question}[theorem]{Question}
\newtheorem{conjecture}[theorem]{Conjecture}

\theoremstyle{remark}

\newtheorem{remark}[theorem]{Remark}
\newtheorem{remarks}[theorem]{Remarks}

\numberwithin{equation}{section}

\newcommand{\exc}{\mathrm{exc}}
\newcommand{\inv}{\mathrm{inv}}

\newcommand{\cosupp}{\mathrm{cosupp}\,}

\newcommand{\ord}{\mathrm{ord}}

\newcommand{\nc}{\mathrm{nc}}
\newcommand{\pp}{\mathrm{pp}}
\newcommand{\cp}{\mathrm{cp}}
\newcommand{\snc}{\mathrm{snc}}
\newcommand{\smooth}{\mathrm{smooth}}
\newcommand{\Aut}{\mathrm{Aut}}
\newcommand{\Frac}{\mathrm{Frac}}
\newcommand{\Spec}{\mathrm{Spec}}
\newcommand{\SL}{\mathrm{SL}}
\newcommand{\SLplus}{\mathrm{SL}_{\mathrm{lex}}^{+}}
\newcommand{\Eold}{E_{\mathrm{old}}}
\newcommand{\Enew}{E_{\mathrm{new}}}

\newcommand{\al}{{\alpha}}
\newcommand{\be}{{\beta}}
\newcommand{\de}{{\delta}}
\newcommand{\ep}{{\varepsilon}}
\newcommand{\De}{{\Delta}}
\newcommand{\ga}{{\gamma}}

\newcommand{\s}{{\sigma}}

\newcommand{\IN}{{\mathbb N}}

\newcommand{\IA}{{\mathbb A}}

\newcommand{\IC}{{\mathbb C}}
\newcommand{\IK}{{\mathbb K}}

\newcommand{\IZ}{{\mathbb Z}}

\newcommand{\cC}{{\mathcal C}}

\newcommand{\cI}{{\mathcal I}}

\newcommand{\cM}{{\mathcal M}}

\newcommand{\cO}{{\mathcal O}}

\newcommand{\cR}{{\mathcal R}}

\newcommand{\tI}{\widetilde{I}}

\newcommand{\ucC}{\underline{\cC}}

\newcommand{\ucM}{\underline{\cM}}

\newcommand{\llb}{{[\![}}
\newcommand{\rrb}{{]\!]}}
\newcommand{\llbr}{{(\!(}}
\newcommand{\rrbr}{{)\!)}}

\newcommand{\RN}[1]{%
  \textup{\uppercase\expandafter{\romannumeral#1}}%
}

\begin{document}
\title[Partial desingularization]{Partial desingularization}
\author[A.~Belotto da Silva]{Andr\'e Belotto da Silva}
\author[E.~Bierstone]{Edward Bierstone}
\author[R.~Ronzon Lavie]{Ramon Ronzon Lavie}
\address[A.~Belotto da Silva]{Universit\'e Paris Cit\'e and Sorbonne Universit\'e, CNRS, IMJ-PRG, F-75013 Paris, France}
\email{andre.belotto@imj-prg.fr}
\address[E.~Bierstone]{University of Toronto, Department of Mathematics, 40 St. George Street, Toronto, ON, Canada M5S 2E4}
\email{bierston@math.utoronto.ca}
\address[R.~Ronzon Lavie]{University of Toronto, Department of Mathematics, 40 St. George Street, Toronto, ON, Canada M5S 2E4}
\email{ramon.ronzon@mail.utoronto.ca}
\thanks{Research supported by Plan d'investissements France 2030, IDEX UP ANR-18-IDEX-0001 (Belotto da Silva), NSERC Discovery Grant RGPIN-2017-06537 (Bierstone) and CONACyT Becas al Extranjero 2018-000009-01EXTF-00250 (Ronzon Lavie)}
\thanks{The authors acknowledge earlier contributions of Pierre Lairez and Franklin Vera Pacheco to the development of ideas in this article}
\date{\today}

\begin{abstract}
We address the following question of \emph{partial desingularization preserving normal crossings}. Given an 
algebraic (or analytic) variety $X$ in characteristic zero, can
we find a finite sequence of blowings-up preserving the normal-crossings locus of $X$, after which the
transform $X'$ of $X$ has only singularities from an explicit finite list of \emph{minimal
singularities}, which we define using the determinants of circulant matrices. In the case of surfaces, for example,
the pinch point or Whitney umbrella is the only singularity needed in addition to normal crossings.

We develop techniques for factorization
(splitting) of a monic polynomial with regular (or analytic) coefficients, satisfying a generic normal crossings
hypothesis, which we use together with resolution of singularities techniques to find local circulant normal
forms of singularities. These techniques in their current state are enough for a positive answer to the question
above, for $\dim X \leq 4$, or in arbitrary dimension if we preserve normal crossings only of order at most three.
In these cases, minimal singularities have smooth normalization.
\end{abstract}

\maketitle

\setcounter{tocdepth}{1}
\tableofcontents

\section{Introduction}\label{sec:intro}
The goal of partial desingularization as described in this article is to understand the nature of the singularities that have to be admitted 
after a sequence of blowings-up $\s: X' \to X$ whose centres are restricted to lie over the complement of the normal crossings locus
$X^\nc$ of an algebraic or analytic variety $X$.

This study is motivated by the following question. Given $X$ (defined over a field $\IK$ of characteristic zero),
can we find a sequence of blowings-up $\s: X' \to X$ such that $\s$ 
preserves the normal crossings locus of $X$, and $X'$ has only normal crossings singularities? Roughly speaking, a variety has 
normal crossings at a point $a$ if it can be defined by a monomial equation 
\begin{equation}\label{eq:nc}
x_1 x_2\cdots x_k = 0
\end{equation}
in local coordinates at $a$.
But the definition of \emph{normal crossings} and the answer to the preceding question depend on the meaning of local coordinates.

\begin{definitions and remarks}\label{def:nc}
We say that an algebraic variety $X$ has \emph{simple normal crossings (snc)} at $a$ if there is an embedding of an
open neighbourhood of $a$ in a smooth variety $Z$, and a system of regular coordinates (or a regular system of parameters)
$(x_1,\ldots,x_n)$ for $Z$ at $a$, with respect to which $X$ is defined by an equation \eqref{eq:nc}. (In this case, we say, more
precisely, that $X$ has \emph{simple normal crossings} $\snc(k)$ of \emph{order} $k$ at $a$.)

Simple normal crossings at $a$ is equivalent to the condition that (the restrictions of $X$ to) all irreducible components 
are smooth (or empty) and transverse at $a$.

We say that an algebraic or analytic variety $X$ has \emph{normal crossings (nc)} at $a$ (or, more precisely, \emph{normal
crossings} $\nc(k)$ of \emph{order} $k$ at $a$) if $X$ is again defined locally by an equation \eqref{eq:nc},
except that here $(x_1,\ldots,x_n)$ is an \'etale
(or local analytic, or formal) coordinate system at $a$ (perhaps after a finite extension of the ground field $\IK$).
In particular, normal crossings is the basic notion in the analytic case.

The \emph{normal crossings locus} $X^\nc$ of $X$ denotes the set of points of $X$ having only normal
crossings singularities. ($X^\nc$ includes all smooth points of $X$.)
\end{definitions and remarks}

\begin{examples}\label{ex:nc}
The nodal curve $y^2 = x^2 + x^3$ has normal crossings but not simple normal crossings at the origin. The curve
$y^2 + x^2 =0$ is nc at $0$, but snc if and only if $\sqrt{-1} \in \IK$. Whitney's umbrella $z^2 - wx^2 = 0$ is nc, but not snc,
at every nonzero point of the $w$-axis $z=x=0$.
\end{examples}

We will take the ground field $\IK$ to be $\IC$ throughout the rest of the article, though 
all results for algebraic varieties hold over any given algebraically closed field 
$\IK$ of characteristic zero.

The answer to the question above is \emph{yes} for snc; see \cite{BDMV}, \cite[Section 12]{BMinv}, \cite[Section 3]{BMmin1}, 
\cite{Kolog}, \cite{Sz}. There are also many interesting variants of the question for snc; for example, \cite{BV1}, \cite{BV2},
\cite{W}. On the other hand, the answer to the question is \emph{no} for nc, in general.

\begin{example}\label{ex:whitney}
The answer is \emph{no} for Whitney's umbrella $X: z^2-wx^2=0$, which has a non-nc
singularity called a \emph{pinch point pp} at $0$. There is no proper birational morphism that eliminates the pinch point
without also modifying nc points, according to the following argument of Koll\'ar \cite{Kolog} (where the question of
Theorem \ref{thm:pp} below and higher-dimensional analogues also was raised). At a nonzero point of the $w$-axis,
$X$ has two local analytic branches. If we go around a small circle about $0$ in the $w$-axis, these branches are interchanged.
This phenomenon continues to hold after any birational morphism that is an isomorphism over the generic point of the $w$-axis.
\end{example}

On the other hand, we have the following result.

\begin{theorem}[\cite{BMmin1}]\label{thm:pp}
For any two-dimensional algebraic (or analytic) variety $X$, there is a morphism $\s: X' \to X$ given by a finite sequence
of smooth blowings-up preserving the normal crossings locus $X^{\nc}$, such that $X'$ has only nc and pp singularities.
\end{theorem}

Whitney's umbrella $X$ has smooth normalization; for example, if we set $w=v^2$, then $z^2-wx^2$ factors as
$(z-vx)(z+vx)$, and the morphism to $X$ from the smooth variety defined by either of the factors is a finite 
birational morphism. See also Proposition \ref{prop:circ}. Normalization plays an important part in classical approaches to resolution
of singularities. In particular, smooth normalization, when it exists, is a relatively simple one-shot method
to resolve singularities.

\begin{conjecture}\label{conj:min}
For any algebraic (or analytic) variety $X$, there is a finite composite of admissible smooth blowings-up 
$\s: X' \to X$ (see Defiinition \ref{def:admiss} following), preserving $X^{\nc}$, such that $X'$ has only 
singularities from an explicit finite list (which we call \emph{minimal singularities}).
\end{conjecture}

In the case of an analytic variety, the morphism $\s$
in Conjecture \ref{conj:min} should be understood to mean a morphism over a given relatively compact open subset of $X$.

\begin{definition}\label{def:admiss}
A smooth blowing-up (i.e., a blowing-up with smooth centre $C$) is \emph{admissible} if
\begin{enumerate}
\item locally, there are
regular coordinates with respect to which $C$ is a coordinate subspace and each component of the 
exceptional divisor $E$ is a coordinate hypersurface (in this case, we say that $C$ and $E$ are \emph{snc});
\item the Hilbert-Samuel function $H_{X,x}$ is locally constant (as a function of $x$) on $C$. 
\end{enumerate}
\end{definition}

In the case that $X$ is a hypersurface (see \S\ref{subsec:split} below), condition (2)
is equivalent to the condition that the order $\ord_x X$ is locally constant on $C$. Definition \ref{def:admiss} corresponds
to the properties satisfied by the blowings-up involved in resolution of singularities in characteristic zero \cite{Hiro},
\cite{BMinv}, \cite{BMfunct}. A reader can safely choose not to focus on (2) since we use the Hilbert-Samuel function
only to reduce the general to the hypersurface case (see \S\ref{subsec:approach} and Section 6).

Minimal singularities will be defined using the determinants of circulant matrices, and include the
\emph{circulant singularities} of \S\ref{subsec:circ} and Section \ref{sec:circ}, which are higher-dimensional
versions of the the pinch point.
The following theorems summarize our general results on the conjecture above.

\begin{theorem}\label{thm:dim4}
Conjecture \ref{conj:min} is true for $\dim X \leq 4$.
\end{theorem}

\begin{theorem}\label{thm:nc(3)}
Let $X^{\nc(3)}$ denote the set of normal crossings points of $X$ of order at most three. Then the analogue
of Conjecture \ref{conj:min} with $X^{\nc}$ replaced by $X^{\nc(3)}$ is true (in any dimension).
\end{theorem}

\begin{remark}\label{rem:min}
In both theorems, the partial desingularization $X'$ of $X$ has smooth normalization. This phenomenon is
particular to $\dim X \leq 4$, or to normal crossings of low order as in Theorem \ref{thm:nc(3)} (see Remark \ref{rem:forward}).
\end{remark}

Conjecture \ref{conj:min} for $\dim X \leq 3$, and the analogue of Theorem \ref{thm:nc(3)} for $X^{\nc(2)}$, are established in \cite{BLMmin2}. The main novelties of the current article are:
\begin{itemize}
\item the development of general techniques for factorization (or splitting) of a monic polynomial with
regular coefficients which satisfies a generic normal crossings hypothesis (Section \ref{sec:split}, see also 
Theorem \ref{thm:splitintro} below);
 \item the use of resolution of
singularities techniques together with such splitting results to obtain normal forms of products of circulant singularities (Sections \ref{sec:lim}
and \ref{sec:triplenc}, see also Theorem \ref{thm:limintro});
\item techniques to move away limits of singularities that occur in a neighbourhood of a circulant point, to
leave only minimal singularities in certain distinguished components of the exceptional divisor (Section \ref{sec:alg}).
\end{itemize}
These techniques are enough to prove Theorems \ref{thm:dim4} and \ref{thm:nc(3)}, as we show in Section \ref{sec:alg}, following an inductive strategy presented in 
\S\ref{subsec:overviewproof}, and discussed in a concluding remark in \S\ref{subsec:concl}. Looking forward,
we mention recent progress on the techniques above in the general case, in Remark \ref{rem:forward}.

We can, in fact, prove more precise versions of Theorems \ref{thm:dim4} and \ref{thm:nc(3)}, for a pair $(X,E)$, where
$E$ is an snc divisor; see Theorems
\ref{thm:mainA}, \ref{thm:mainB} and \ref{thm:main3}.

\begin{remark}\label{rem:minimal}
The term \emph{minimal singularities} comes from \cite{BMmin1}; although the resemblance to ``minimal''
in the minimal model program is not coincidental, the meaning is not the same. Minimal singularities may be 
compared also to the singularities of the image a generic morphism of smooth varieties $X \to Z$, where $\dim Z = \dim X +1$
(see, for example, \cite{Rob}), or to the singularities of stable mappings of differentiable manifolds $X \to Z$, $\dim Z = \dim X +1$.
The notions coincide if $\dim X \leq 2$, but not in general.
\end{remark}

The resolution of singularities techniques used in the article involve the desingularization invariant $\inv$ of \cite{BMinv}, \cite{BMfunct}.
As an illustration of our use of these techniques, let us sketch a proof of Theorem \ref{thm:pp}.

\begin{proof}[Proof of Theorem 1.4]
We consider a hypersurface $X$ in 3 variables. Then the triple normal crossings $\nc(3)$ points of $X$ are isolated,
and the $\nc(2)$ locus has codimension two in the ambient smooth variety
(codimension one in $X$). We can blow up with smooth centres in the
complement of $\nc(3)$, without modifying $\nc(2)$, until the maximal value of the desingularization invariant $\inv$
equals the value $\inv(\nc(2))$ that it takes at an $\nc(2)$ point. Then the locus of points where $\inv = \inv(\nc(2))$
is a smooth curve $C$ in the strict transform of $X$.

A basic understanding of the desingularization invariant (which we will recall and use in the article) shows that,
at any point of $C$, we can choose local coordinates in which (the strict transform of) $X$ is given by an equation
\begin{equation}\label{eq:prepp}
z^2 - w^k x^2 = 0,
\end{equation}
where $w$ is an exceptional divisor; then $X$ is $\nc(2)$ on $\{z=x=0,\, w\neq 0\}$.

Then, by finitely many blowings-up with centre $\{z=w=0\}$, we can transform $X$ to either
\begin{align*}
z^2 - x^2 &= 0 \qquad \nc(2)\\
\text{or}\quad z^2 - wx^2 &= 0 \qquad \text{pp}
\end{align*}
(according as $k$ is even or odd).

(Note that, in any case, $z^2 - w^k x^2$ splits as a polynomial in $w^{1/2}, x, z$.)

The exponent $k$ appearing in \eqref{eq:prepp} is, in fact, a local invariant of $X$, and the preceding blowings-up 
defined in local coordinates extend to global admissible blowings-up (see Theorem \ref{thm:limintro} 
and \S\ref{subsec:clean}).
\end{proof}

\subsection{Circulant singularities}\label{subsec:circ}
Our minimal singularities, in general, are products of \emph{circulant singularities}, described in detail in Section \ref{sec:circ}
following (see also \eqref{eq:cpintro} below),
together with their \emph{neighbours}. (For example, given any singularity that has to be admitted after
blowing-up sequences preserving nc, any neighbouring singularity must also be admitted. See also Section \ref{sec:alg}.)

Circulant singularities were introduced in \cite{BMmin1},
\cite{BLMmin2} (where they were called \emph{cyclic singularities}); we give a description in terms of circulant
matrices (suggested by Franklin Vera Pacheco) in Section \ref{sec:circ}, which is convenient for studying their branching behaviour.
A circulant singularity $\cp(k)$ of order $k$ is a singularity which must be admitted as a limit of $\nc(k)$, after
a blowing-up sequence preserving normal crossings. In particular, $\pp = \cp(2)$ and $\smooth = \cp(1)$. The circulant singularity
$\cp(k)$ is the singularity at the origin of the hypersurface in $\geq k+1$ variables given by
$$
\De_k(x_0, w^{1/k}x_1, \ldots, w^{(k-1)/k}x_{k-1}) = 0,
$$
where
\begin{equation}\label{eq:cpintro}
\De_k(X_0, X_1, \ldots, X_{k-1}) = \prod_{\ell = 0}^{k-1}\, (X_0 + \ep^\ell X_1 + 
                                     \cdots + \ep^{(k-1)\ell}X_{k-1})
\end{equation} 
with $\ep = e^{2\pi i/k}$, is the determinant of the circulant matrix in $k$ indeterminates
$(X_0,X_1,\ldots,X_{k-1})$; see \eqref{eq:detcirc}. For example,
\begin{align*}
\cp(2) = \pp: &\qquad \De_2(z,w^{1/2}x) = z^2 -wx^2,\\
\cp(3): &\qquad \De_3(z, w^{1/3}y, w^{2/3}x) = z^3 + wy^3 + w^2x^3 - 3wxyz.
\end{align*}

\begin{examples}\label{ex:min}
(1) \emph{Minimal singularities in 4 variables}; i.e., $\dim X =3$ \cite{BLMmin2}.
The complete list of minimal singularities in 4 variables comprises $\cp(3)$ and its (singular) neighbours,
together with $\nc(4)$, $\cp(2)$ and $\smooth \times \cp(2)$, where the latter means product
as ideals; i.e., $y(z^2-wx^2) = 0$.
The neighbours of $\cp(3)$ are $\nc(2)$, $\nc(3)$, and the following singularity of order 2:
$$
\De_3(z, w^{1/3}y, w^{2/3}) = 0.
$$
The latter was called a \emph{degenerate pinch point} in \cite{BLMmin2}.

\smallskip
In general, the minimal singuarities in $n+1$ variables include all those which occur in $\leq n$ variables
(understood as formulas in $n+1$ variables where not all variables appear), together with $\nc(n+1)$ and
all singularities in small neighbourhoods of products of circulant singularities
that make sense a limits of $\nc(k)$, $k = n$ (see Theorem \ref{thm:limintro}). 
But the following shows that this list is not exhaustive.

\medskip\noindent
(2) \emph{Minimal singularities in 5 variables}; i.e., $\dim X = 4$ (see Section \ref{sec:alg}).
Minimal singularities in 5 variables include the following limits of 4-fold normal crossings $\nc(4)$: 
$\cp(4)$, $\smooth \times \cp(3)$, $\cp(2) \times \cp(2)$,
$\smooth \times \smooth \times \cp(2)\, =\, \nc(2) \times \cp(2)$.
The circulant singularity $\cp(4)$ is the vanishing locus of 
$$
\De_4(x_0, w^{1/4}x_1, w^{2/4}x_2, w^{3/4}x_3).
$$
\smallskip\noindent
Following are the \emph{neighbours} of $\cp(4)$:
\begin{align*}
\text{(1)} &\qquad \De_4(x_0, w^{1/4}x_1, w^{2/4}x_2, w^{3/4})\,,\\
\text{(2)} &\qquad  \De_4(x_0, w^{1/4}x_1, w^{2/4}, w^{3/4}x_3)\,,\\
\text{($2'$)} &\qquad  \De_4(x_0, w^{1/4}x_1, w^{2/4}, w^{3/4}x_2x_3)\,,\\
\text{(3)} &\qquad  \De_4(x_0, w^{1/4}x_1, w^{2/4}, w^{3/4})\,.
\end{align*}
Items (1), (2) and (3) in this list are the non-$\nc$ singularities in an arbitrarily small
neighbourhood of $\cp(4)$ (except for the latter itself), while ($2'$) illustrates
a phenomenon that does not appear in fewer than 5 variables; ($2'$) has to
be admitted as a limit
of singularities of the form (2). In ($2'$), $x_2$ is an exceptional divisor. For
details, see Section \ref{sec:alg}.
\end{examples}

\subsection{Approach to the main problems}\label{subsec:approach}
We use the desingularization invariant $\inv$ and the resolution of singularities algorithm
of \cite{BMinv}, \cite{BMfunct} to the reduce our main problems to a study of the
singularities of a hypersurface $X$ near a point
in the closure of the $\nc(k)$-locus, for given $k$, where $X$ has a convenient description
in suitable local \'etale or analytic coordinates.

Normal crossings singularities are singularities of hypersurfaces. We say that
$X$ is a \emph{hypersurface} if, locally, $X$ can be defined by a principal ideal on
a smooth variety. (We say that $X$ is an \emph{embedded hypersurface} if $X
\hookrightarrow Z$, where $Z$ is smooth and $X$ is defined by a principal ideal on $Z$.)
Conjecture \ref{conj:min} can be reduced to the case of a hypersurface using 
\cite{BMinv, BMfunct}. Indeed, the desingularization algorithm of
these articles involves blowing up
with smooth centres in the maximum strata of the Hilbert-Samuel function. The latter
determines the local embedding dimension, so the algorithm first eliminates points of
embedding codimension $> 1$ without modifying nc points.  
(Recall that if $H$ is the Hilbert-Samuel function of the local ring
of a variety at a given point $a$, then the minimal embedding dimension
at $a$ is $H(1)-1$.)

Let $X \hookrightarrow Z$ denote an embedded hypersurface, $\dim Z = n$.
Then, for any $k\leq n$,  the $\nc(k)$-locus of $X$ is a smooth subspace
of $X$ of codimension $k$ in $Z$. 

The desingularization invariant $\inv$ is upper-semicontinuous 
with respect to the lexicographic ordering,
and the locus of points where $\inv$ takes a given value is smooth.
The value $\inv(\nc(k))$ of $\inv$ at a normal crossings point of order $k$ (in \emph{year zero};
i.e., before we start blowing up) is
\begin{equation}\label{eq:invnc}
\inv(\nc(k)) = (k,0,1,0,\ldots,1,0,\infty),
\end{equation}
where there are $k$ pairs before $\infty$. 

We remark that the condition $\inv(a) = \inv(\nc(k))$ does not, in general, imply that $X$
is nc at $a$. For example, if $X$ is the affine variety $x_1^k + \cdots + x_k^k = 0$,
then $\inv(a) = (k,0,1,0,\ldots,1,0,\infty)$, where there are $k$ pairs, but $X$ is not $\nc$
at $0$ if $k>2$ .

More details of $\inv$ and the desingularization algorithm will be recalled in \S\ref{subsec:inv}.
We also refer the reader to \cite{BMfunct} and to the \emph{Crash course on the desingularization invariant}
\cite[Appendix A]{BMmin1}. Note, in particular, that $\inv$ is defined recursively over a sequence of
admissible blowings-up in the desingularization algorithm. In \emph{year} $j$ (i.e., after $j$ blowings-up), in general,
$\inv$ depends on the previous blowings-up; it is not simply the year zero $\inv$ computed as if year $j$
were year zero.

\smallskip
Our approach to the main problems involves a general inductive or recursive strategy which is
the subject of Section \ref{sec:alg} below; it is based on
the following inductive formulation of Conjecture \ref{conj:min}:
given $k$, there is a finite composite of admissible smooth blowings-up
$\s: X' \to X$, preserving normal crossings of order up to $k$, such that $X'$ has only minimal singularitites. 
Theorems \ref{thm:dim4} and \ref{thm:nc(3)} both follow from this assertion. We will,
in fact, need a more precise version of the assertion for a pair $(X,E)$, where $X\hookrightarrow Z$ is an
embedded hypersurface and $E\subset Z$ is an snc divisor; see Section \ref{sec:alg}.

The inductive step of the argument can be described roughly in the following way (though it is needed
actually in the context of a pair $(X,E)$; see Claim \ref{claim:ind}).
\begin{itemize}
\item Blow up following the desingularization algorithm until $\inv\leq \inv(\nc(k))$. Then 
the locus of points where $\inv = \inv(\nc(k))$ is a smooth
closed subspace $S_k$ of codimension $k$ in $Z$. We can further blow up to eliminate
any component of $S_k$ on which $X$ is not generically $\nc(k)$.
\item Modify non-nc points of $S_k$ to get minimal singularities in $\Sigma_k = S_k \cup D_k$, where
$D_k$ is a distinguished subset of the exceptional divisor, and only normal crossings in $U\backslash\Sigma_k$, for
some neighbourhood $U$ of $\Sigma_k$.
\item Apply the inductive hypothesis in the complement of $\Sigma_k$; the centres of blowing up are closed in $X$ in the
case of $S_k$, or extend to
to global admissible centres when we deal with a pair $(X,E)$, in general.
\end{itemize}

Modification of the stratum $S=S_k$ itself involves three main steps: \emph{splitting}, described in
Theorem \ref{thm:splitintro} following, \emph{cleaning} to get circulant normal form as in Theorem \ref{thm:limintro},
and \emph{moving away} limits of singularities in a neighbourhood of a (product) circulant point to leave only
minimal singularities in the distinguished divisor $D_k$. The latter step will be described in Section \ref{sec:alg}.

\subsection{Splitting techniques and circulant normal form}\label{subsec:split}
The non-$\nc(k)$ points of $X$ in $S = S_k$ form a proper closed subspace $T$ (see Lemma \ref{lemma:genericnc}).
After resolving the singularities of $T$ if necessary, we can assume that, given $a\in S$,
we can choose \'etale (or analytic) local coordinates
\begin{equation*}
(w,u,x,z) = (w_1,\ldots,w_r, u_1,\ldots,u_q, x_1,\ldots,x_{k-1},z)
\end{equation*}
for $Z$ at $a$, in which $X$ is given by $f(w,u,x,z)=0$, where
\begin{equation}\label{eq:weier}
f(w,u,x,z) = z^k + a_1(w,u,x)z^{k-1} +\cdots + a_k(w,u,x),
\end{equation}
the coefficients $a_i(w,u,x)$ are regular (or analytic) functions, $S = \{z=x=0\}$,
the exceptional divisor is $w_1\cdots w_r=0$, and
the complement of $\{z=x=0,\,w_1\cdots w_r = 0\}$ maps isomorphically onto the
original set of $\nc(k)$ points of $X$ in $S$.
(It follows that every coefficient $a_i$ vanishes to order at least $i$ with respect
to $(x,z)$ at $a$.) 

We are interested in the splitting or factorization of $f$ at $a$ as
\begin{equation}\label{eq:split}
f(w,u,x,z) = \prod_{j=1}^k \left(z- b_j(w,u,x)\right),
\end{equation}
where each $b_j$ belongs to the ideal generated by $x_1,\ldots,x_{k-1}$.
For example, at an $\nc(k)$ point, there is a formal splitting \eqref{eq:split},
where each $b_j$ has order $1$.

From the latter generic splitting condition, it follows (at least in the algebraic case) that there is a unique
splitting of $f$ in $\overline{\IC(w)}\llb u,x\rrb [z]$, where each $b_j(u,w,x) \in 
\overline{\IC(w)}\llb u,x\rrb$. Here $\overline{\IC(w)}$ denotes an algebraic
closure of the field of fractions $\IC(w)$ of the polynomial ring $\IC[w]$
(see \S\ref{subsec:basicsplit}).

For example, if there is a single $w$ variable, then $f$ splits over $\IC(w^{1/p})\llb u,x\rrb$,
for some $p$, by the Newton-Puiseux theorem and elementary Galois theory, 
and we can take $p=k$, if $f$ is irreducible (see Lemmas \ref{lemma:finitesplit} and \ref{cor:finitesplit}).

Following is a simple basic example which illustrates Theorem \ref{thm:splitintro} following, and also
shows that the conclusion in this result cannot, in general, be strengthened.

\begin{example}\label{ex:basic} Let
$$
f(w,x,z) = z^2 + (w^3 + x) x^2.
$$
Then $f$ (or the subvariety $X$ of $\IA_{\IC}^3$ defined by $f(w,x,z)=0$) is $\nc(2)$ at
every nonzero point of the $w$-axis $\{x=z=0\}$. The function $f$ does not split over $\IC\llb w,x\rrb$,
but we can write
$$
f(v^2,x,z) = z^2 + v^6\left(1+\frac{x}{v^6}\right) x^2,
$$
so that $f(w,x,z)$ splits in $\IC(w^{1/2})\llb x\rrb [z]$. 

Note that $f(v^2,x,z)$ is not normal crossings
at $0$ as a formal power series in $\IC\llb v,x,z\rrb$, but it is normal crossings in $\IC(v)\llb x,z\rrb$
(i.e., as a formal power series in $(x,z)$ with coefficients in the field $\IC(v)$).

Consider the blowing-up $\s$ of the origin in $\IA_{\IC}^3$. The $w$-axis lifts to the $w$-chart of $\s$,
given by substituting $(w,wx,wz)$ for $(w,x,z)$, and the strict transform of $X$ is 
given by $f' = 0$ in the $w$-chart, where
$$
f'(w,x,z) := w^{-2} f(w,wx,wz) = z^2 + w(w^2 + x)x^2.
$$
After two more blowings-up of the origin, we get
$$
f'(w,x,z) = z^2 + w^3(1 + x)x^2, 
$$
so that $f'(w,x,z)$ splits over $\IC\llb w^{1/2},x\rrb$ (or $f'(v^2,x,z)$ splits in an \'etale neighbourhood of the origin).

After an additional \emph{cleaning blowing-up}, with centre $\{z=w=0\}$, we get a pinch point.
\end{example}

\begin{theorem}[limits of $\nc(k)$ in $n= k+1$ variables]\label{thm:splitintro} Let
\begin{equation}\label{eq:splitintro}
f(w,x_1,\ldots,x_{k-1},z) = z^k + a_1(w,x)z^{k-1} +\cdots + a_k(w,x),
\end{equation}
where the coefficients $a_i(w,x)$ are regular (or analytic) functions. If $f(w,x,z)$ is $\nc(k)$ on $\{z=x=0,\,w\neq0\}$,
then, after a finite number of blowings-up of $0$, $f$ splits over $\IC\llb w^{1/p}, x\rrb$, for some positive integer $p$.
\end{theorem} 

Theorem \ref{thm:splitintro} is proved in Section \ref{sec:split} using the splitting over $\overline{\IC(w)}\llb x\rrb$
together with a multivariate Newton-Puiseux theorem due to Soto and Vicente \cite{SV}, to show that the powers
of $w$ in the denominators of the roots are bounded linearly with respect to the degree with respect to $x$ in the numerators.

\begin{question}\label{conj:opt}
Consider the general case,
\begin{multline*}
f(w_1,\ldots,w_r, u_1,\ldots,u_q, x_1,\ldots,x_{k-1},z) \\= z^k + a_1(w,u,x)z^{k-1} +\cdots + a_k(w,u,x),
\end{multline*}
where $f(w,u,x,z)$ is is $\nc(k)$ on $\{z=x=0,\,w_1\cdots w_r \neq0\}$. Is it true that, after finitely many blowings-up with
successive centres of the form $\{z=x=w_j =0\}$, for some $j$, $f$ splits over
$\IC\llb u, w^{1/p}, x\rrb$, for some $p$, where $w^{1/p} := (w_1^{1/p},\ldots, w_r^{1/p})$?
\end{question}

We give a positive answer to this question in the case $k\leq 3$, under the additional hypothesis $\inv(0) = \inv(\nc(k))$; see Proposition \ref{prop:limnc3}.

Theorem \ref{thm:limintro}  following ties together the splitting theorem \ref{thm:splitintro} with the notion of circulant singularity.

\begin{remark}\label{rem:memory}
In Theorem \ref{thm:limintro} and throughout the article, it is convenient to continue to use the same notation $X$ instead of,
for example, $X_j$ for the strict transform of $X = X_0$ after $j$ blowings-up.
\end{remark}

\begin{theorem}[circulant normal form]\label{thm:limintro}
Consider an embedded hypersurface $X\hookrightarrow Z$, as above, and assume that
$n := \dim Z = k+1$. Let $U$ denote an open subset of $Z$. Assume that (after a sequence of $\inv$-admissible
blowings-up of $U$; cf. \S\ref{subsec:inv}) the maximum value of $\inv$
on $U$ is $\inv(\nc(k))$ and that $X$ is generically $\nc(k)$ on the stratum $S := \{\inv = \inv(\nc(k))\}$ in $U$; in particular,
$S$ is a smooth curve in $U$. Then there is a finite sequence of admissible blowings-up of $U$ (in fact, admissible
for the truncated invariant $\inv_1$; see \S\ref{subsec:inv}),
preserving the $\nc(k)$-locus, after which $X$
is a product of circulant singularities at every point of $S$; i.e., $X$ can be defined locally at every point
of $S$ by an equation of the form
\begin{equation}\label{eq:limintro}
\prod_{i=1}^s \De_{k_i} \left(y_{i0}, w^{1/k_i} y_{i1}, \ldots, w^{(k_i-1)/k_i}y_{i,k_i-1}\right) = 0,
\end{equation}
in suitable \'etale (or local analytic) coordinates\, $\left(w, (y_{i\ell})_{\ell=0,\ldots,k_i -1,\, i=1,\ldots s}\right)$,
where $k_1 + \cdots + k_s = k$.
\end{theorem}

The equation \eqref{eq:limintro} defines a variety with smooth normalization.

Note that $\nc(k)$ is itself a product of circulant singularities (each of order $1$). Theorem \ref{thm:limintro}
is proved in Section \ref{sec:lim}. A proof of
Conjecture \ref{conj:min} following our approach requires analogues of
Theorems \ref{thm:splitintro} and \ref{thm:limintro} for $k+1 < n$; we prove these results
for $k\leq 3$ in Section \ref{sec:triplenc}. The techniques of Sections \ref{sec:split},
\ref{sec:lim} and \ref{sec:triplenc} are put together in Section \ref{sec:alg}, for modification of the strata
$S_k$ as described in \S\ref{subsec:approach} above. An overview the the proofs of Theorems 
\ref{thm:dim4} and \ref{thm:nc(3)} is given in \S\ref{subsec:overviewproof}.

\begin{remark}\label{rem:forward}
\emph{Looking forward.} A positive answer to Question \ref{conj:opt} and a general
version of Theorem \ref{thm:limintro} have been recently obtained by the first two authors, and are
planned for a forthcoming article. The extension of Theorem \ref{thm:limintro} requires a more
general formulation of the idea of a product of circulant singularities \eqref{eq:limintro}; in particular,
the analogous factors may include fractional powers of more than a single
exceptional variable $w$, so the normalization need no longer be smooth.
\end{remark}

\medskip
\section{Circulant singularities}\label{sec:circ}
Circulant singularities provide a generalization to arbitrary dimension of the \emph{pinch point}
singularity that occurs at the origin of Whitney's umbrella $z^2-w y^2 = 0$.

\medskip
Given indeterminates $X = (X_0,X_1,\ldots,X_{k-1})$, we define the \emph{circulant matrix}
\medskip
\begin{equation}\label{eq:circ}
C_k(X_0,X_1,\ldots,X_{k-1}) :=
\begin{pmatrix}
X_0 & X_1 & \cdots & X_{k-1} \\
X_{k-1} & X_0 & \cdots & X_{k-2} \\
\vdots & \vdots & \ddots & \vdots \\
X_1 & X_2 & \cdots & X_0
\end{pmatrix}.
\end{equation}
See \cite{KS} for a nice introduction to circulant matrices.

\smallskip
The circulant matrix $C_k(X_0,X_1,\ldots,X_{k-1})$ has eigenvectors
$$
V_\ell = (1,\, \ep^\ell,\, \ep^{2\ell},\, \ldots,\, \ep^{(k-1)\ell}),
$$
$\ell = 0, \ldots, k-1$, where $\ep= e^{2\pi i/k}$. The corresponding
eigenvalues are
\begin{equation}\label{eq:eigenval}
Y_\ell = X_0 + \ep^\ell X_1 + \cdots + \ep^{(k-1)\ell}X_{k-1},\quad \ell = 0, \ldots, k-1.
\end{equation}

Let $\De_k$ denote the determinant $\det C_k$. Then
\begin{equation}\label{eq:detcirc}
\begin{aligned}
\De_k (X_0, \ldots, X_{k-1}) &=  Y_0 \cdots Y_{k-1}\\
               &= \prod_{\ell = 0}^{k-1}\, (X_0 + \ep^\ell X_1 + \cdots + \ep^{(k-1)\ell}X_{k-1}).
\end{aligned}
\end{equation}

Given indeterminates $(w,x_0,\ldots,x_{k-1})$, set
\begin{align}\label{eq:circfact}
P_k (w,x_0,\ldots,x_{k-1}) :=&\, \De_k (x_0,\, w^{1/k}x_1,\, \ldots,\, w^{(k-1)/k}x_{k-1})\\
                               =&\, \prod_{\ell = 0}^{k-1}\, (x_0 + \ep^\ell w^{1/k} x_1 + 
                                     \cdots + \ep^{(k-1)\ell}w^{(k-1)/k}x_{k-1})\nonumber
\end{align}
Then $P_k (w,x_0,\ldots,x_{k-1})$ is an irreducible polynomial.
We define the \emph{circulant} or \emph{circulant point} singularity $\text{cp}(k)$ as the
singularity at the origin of the variety $X$ defined by the equation
$P_k (w,x_0,\ldots,x_{k-1}) = 0$; i.e., by the equation
$$
\De_k (x_0,\, w^{1/k}x_1,\, \ldots,\, w^{(k-1)/k}x_{k-1}) = 0.
$$
(In \cite{BMmin1, BLMmin2}, a circulant point is called a ``cyclic point''.)

For example, $\text{cp}(2)$ is the pinch point, and $\text{cp}(3)$ is given by
\begin{equation}\label{eq:cp3expansion}
P_3 (w,z,y,x) = z^3 + wy^3 + w^2 x^3 - 3wxyz.
\end{equation}

\begin{proposition}\label{prop:circ}
Circulant singularities have smooth normalization.
\end{proposition}

\begin{proof}
If we set $w=v^k$, then $P_k (w,x_0,\ldots,x_{k-1})$ factors as
$$
\prod_{\ell = 0}^{k-1}\, (x_0 + \ep^\ell v x_1 + 
                                     \cdots + \ep^{(k-1)\ell}v^{k-1}x_{k-1}),
$$
and the morphism $\nu$ to $X$ of the smooth hypersurface defined by any of the factors
is a finite birational morphism. Therefore, $\nu$ is the normalization of $X$ (up to isomorphism);
\cite[{\S}III.8,\,Thm.\,3]{Mum}.
See Corollary \ref{cor:lairez} below
for an elementary proof of the proposition.
\end{proof}

\medskip

\begin{remark}\label{rem:circ}
We rewrite \eqref{eq:eigenval},
\medskip
\begin{equation}\label{eq:eigenmatrix}
\left(\begin{matrix} Y_0\\
                              Y_1\\
                              Y_2\\
                              \vdots\\
                              Y_{k-1}
\end{matrix}\right) 
= 
 \left(\begin{matrix} 1 & 1 & 1& \cdots & 1\\
                               1 & \ep^1 & \ep^2 & \cdots & \ep^{k-1}\\
                               1 & \ep^2 & \ep^4 & \cdots & \ep^{2(k-1)}\\
                               \vdots & \vdots & \vdots & \ddots & \vdots\\
                               1 & \ep^{k-1} & \ep^{2(k-1)} & \cdots & \ep^{(k-1)^2}
 \end{matrix}\right)
 \left(\begin{matrix} X_0\\
                              X_1\\
                              X_2\\
                              \vdots\\
                              X_{k-1}
 \end{matrix}\right)                       
\end{equation}\\
The rows (and the columns) of the matrix in \eqref{eq:eigenmatrix} are the eigenvectors $V_0,\ldots,V_{k-1}$.

\medskip
Recall that 
$$
\sum_{l=0}^{k-1} \ep^{i\ell} = \begin{cases} k, \quad i=0,\\
                                                                    0, \quad i=1,\ldots, k-1.
                                              \end{cases}                      
$$
The inverse of the linear transformation \eqref{eq:eigenmatrix} is
\medskip
\begin{equation}\label{eq:inveigenmatrix}
\left(\begin{matrix} X_0\\
                              X_1\\
                              X_2\\
                              \vdots\\
                              X_{k-1}
\end{matrix}\right) = \,\frac{1}{k}\, \left(\begin{matrix} 1 & 1 & 1& \cdots & 1\\
                               1 & \ep^{k-1} & \ep^{k-2} & \cdots & \ep^{1}\\
                               1 & \ep^{2(k-1)} & \ep^{2(k-2)} & \cdots & \ep^{2}\\
                               \vdots & \vdots & \vdots & \ddots & \vdots\\
                               1 & \ep^{(k-1)^2} & \ep^{(k-1)(k-2)} & \cdots & \ep^{k-1}
 \end{matrix}\right)
 \left(\begin{matrix} Y_0\\
                              Y_1\\
                              Y_2\\
                              \vdots\\
                              Y_{k-1}
\end{matrix}\right)             
\end{equation}
\end{remark}

\medskip
\section{Splitting results}\label{sec:split}
\subsection{Basic splitting lemmas}\label{subsec:basicsplit}
Let $\IC(w)$ denote the field of fractions of the polynomial ring $\IC[w] = \IC[w_1,\ldots,w_r]$. 
Let $\IC\llbr w\rrbr$ denote the field of fractions of the formal power series ring $\IC\llb w\rrb = \IC\llb w_1,\ldots,w_r\rrb$,
and let $\overline{\IC\llbr w\rrbr}$ denote an algebraic closure of $\IC\llbr w\rrbr$. An algebraic closure
$\overline{\IC(w)}$ of $\IC(w)$ is given by the subfield of $\overline{\IC\llbr w\rrbr}$ consisting of elements
that are algebraic over $\IC(w)$ (or over $\IC[w]$).

In a single variable $w$, $\IC\llbr w\rrbr$ is the field of formal Laurent series in $w$ over $\IC$ (with finitely many
negative exponents), and
$\overline{\IC\llbr w\rrbr}$ is given by the field of formal Puiseux series in $w$ over $\IC$; i.e., formal Laurent
series over $\IC$ in $w^{1/k}$, where $k$ ranges over the nonnegative integers. 
Since $\overline{\IC\llbr w\rrbr} = \bigcup_{k\in\IN} \IC\llbr w^{1/k} \rrbr$, any finite extension of the field $\IC\llbr w\rrbr$ 
lies in $\IC\llbr w^{1/k} \rrbr$, for some $k$.

Consider a monic polynomial
\begin{equation}\label{eq:basicsplit1}
\begin{aligned}
f(w,y,z) &= f(w_1,\ldots,w_r,y_1,\ldots,y_m,z)\\
            & = z^k + a_1(w,y)z^{k-1} + a_2(w,y)z^{k-2}+\cdots + a_k(w,y)
\end{aligned}
\end{equation}
in $z$ with coefficients $a_i(w,y)$ which are regular functions at $0 \in \IC^{r+m}$ (i.e., rational functions
with nonvanishing denominators in a fixed common neighbourhood of $0$). We say that $f$ \emph{splits formally}
at a point $(w,y,z)=(w_0,y_0,0)$ (or $f$ \emph{splits} in $\IC\llb w-w_0, y-y_0\rrb [z]$, or $f$ \emph{splits over}
$\IC\llb w-w_0, y-y_0\rrb$) if $f$, considered as a formal expansion at $(w_0,y_0,0)$ (or as an expansion
in $\IC\llb w-w_0, y-y_0\rrb [z]$) factors as 
\begin{equation}\label{eq:basicsplit2}
f(w,y,z) = \prod_{j=1}^k (z - b_j(w,y)),
\end{equation}
where, for each $j$, $b_j(w,y) \in \IC\llb w-w_0, y-y_0\rrb$ and $b_j(w,y)$ vanishes when $(w-w_0, y-y_0) = (0,0)$.

Analogously, we can consider splitting in
$\overline{\IC\llbr w\rrbr}\llb y-y_0\rrb [z]$, etc.

\begin{lemma}\label{lemma:basicsplit}
Consider $f(w,y,z)$ as in \eqref{eq:basicsplit1}.
Suppose that $f$ splits formally at a point $(w,y,z)=(w_0,y_0,0)$. 
Then $f$ splits in $\overline{\IC(w)}\llb y-y_0\rrb [z]$.
\end{lemma}

\begin{proof}
We can assume that $y_0 = 0$. 
There is an isomorphism of $\overline{\IC(w)}$ with $\overline{\IC(w-w_0)}$ induced by the
isomorphism $w \mapsto w_0 + (w-w_0)$ of $\IC[w]$ to $\IC[w-w_0]$, so it is enough to show that 
$f$ splits in $\overline{\IC(w-w_0)}\llb y\rrb [z]$.

The roots $b_j(w,y) \in \IC\llb w-w_0, y\rrb$ are algebraic over $\IC[w-w_0,y]$. The result follows
since algebraicity is preserved by partial differentiation and by evaluation (i.e., by setting $z=0$, $y = 0$).
\end{proof}

\begin{remark}\label{rem:basicsplitAnalytic}
In the analytic case, assume that $f(w,y,z) \in \cO(W \times U)[z]$, where $W$ and $U$ are open subsets
of $\IC^r$ and $\IC^m$ (respectively). Then Lemma \ref{lemma:basicsplit} still holds, with the conclusion
$f \in \overline{\Frac(\cO(W)}\llb y-y_0\rrb [z]$, where $\Frac$ denotes the field of fractions.

Indeed, we can extend all results of this section to the analytic case by replacing the ring $\overline{\IC(w)}$ with 
$\overline{\Frac(\cO(W))}$.
\end{remark} 
 
 \begin{remark}\label{rem:basicsplit}
We will be interested in Lemmas \ref{lemma:basicsplit} and \ref{lemma:finitesplit} following
in a situation where $y = (u,x) = (u_1,\ldots, u_q, x_1,\ldots, x_{k-1})$,
$y_0 = (u_0, 0)$, the vanishing locus $\{w_1\cdots w_r =0\}$ represents an exceptional divisor, and $f(w,u,x,z)$ in $\nc(k)$ at every
point of $\{z=x=0,\, w_1\cdots w_r \neq 0\}$ (see \ref{subsec:split}).
\end{remark}

\begin{lemma}\label{lemma:finitesplit}
Consider $f(w,y,z)$ as in \eqref{eq:basicsplit1}.
Suppose that $f$ splits in $\overline{\IC(w)}\llb y-y_0\rrb [z]$. Then there is a finite and normal
extension $L$ of $\IC(w)$ in $\overline{\IC(w)}$
such that $f$ splits in $L\llb y-y_0\rrb[z]$.
\end{lemma}

\begin{proof} 
We can assume that $y_0 = 0$.
By the hypothesis, $f$ splits in $\overline{\IC(w)}\llb y\rrb[z]$ as
\begin{equation}\label{eq:Wier}
f = \prod_{j=1}^{k'}g_j^{m_j},\qquad g_j(w,y,z) = z - b_j(w,y),\ j=1,\ldots,k',
\end{equation}
where $k'\leq k$, each $m_j $ is a positive integer, and the $b_j(w,y)$ are distinct
elements of $\overline{\IC(w)}\llb y\rrb$.
(The decomposition in this form is unique.) 

Consider the formal expansions
\begin{equation*}
b_j(w,y) = \sum_{\ga\in\IN^m} b_{j,\ga} y^\ga,\quad j=1,\ldots,k',
\end{equation*}
where the coefficients $b_{j,\ga}\in  \overline{\IC(w)}$.
Set $M :=  \overline{\IC(w)}$ and let $L$ denote the subfield of $M$
generated over $\IC(w)$ by the $b_{j,\ga}$, $j=1,\ldots,k'$, $\ga\in \IN^{m}$. We will show that
$L$ is a normal extension of $\overline{\IC(w)}$ with finite automorphism group, and therefore a finite extension.

First, consider $\s \in \Aut_{\IC(w)} M$, where the latter denotes the group of field automorphisms over $M$ over $\IC(w)$.
We claim that $\s L = L$; i.e., $\s$ induces an automorphism of $L$ over $\IC(w)$. Indeed, the action of $\s$ on $M$
extends to an action on $\overline{\IC(w)}\llb y\rrb[z]$ which fixes $f$ but permutes the elements $g_j$, by uniqueness
of the decomposition \eqref{eq:Wier}. Therefore, $\s$ fixes the \emph{set} $\{b_{j,\ga}: j=1,\ldots,k',\, \ga\in \IN^{m}\}$
(not the elements of this set). In other words, $\s L = L$.

We claim, moreover, that $L$ is a normal extension of $\IC(w)$; i.e., any irreducible polynomial $p(t) \in \IC(w)[t]$ which
has a root $a_1$ in $L$, splits in $L$. First of all, the algebraic closure $M = \overline{\IC(w)}$ is trivially a normal extension
of $\IC(w)$, so that $\IC(w)$ is the fixed point set of $\Aut_{\IC(w)} M$, by the fundamental theorem of Galois theory.
Now, $\Aut_{\IC(w)} M$ maps to a subgroup $S$ of the permutation group of the roots of $p(t)$, and $\prod_{\tau\in S}(t-\tau(a_1))$
is fixed by $\Aut_{\IC(w)} M$, so it is a polynomial over $\IC(w)$. This polynomial cannot be a nontrivial factor of $p$
because $p$ is irreducible, so we get the claim.

Now consider $\s \in \Aut_{\IC(w)} L$. As above, $\s$ induces a permutation of the $g_j$.
Moreover, $\s$ is determined by its action on $\{b_{j,\ga}\}$, and therefore by its action on the $g_j$; i.e.,
$\Aut_{\IC(w)} L$ embeds as a subgroup of the finite group of permutations of  $\{g_j\}$, as required.
\end{proof}

\begin{lemma}\label{cor:finitesplit}
Assume that $w$ is a single variable.
\begin{enumerate}
\item With the hypotheses of Lemma \ref{lemma:finitesplit},
$f(w,y-y_0,z)$ splits in $\IC(w^{1/p})\llb y-y_0\rrb[z]$, for some $p$.
\item Suppose $f(w,y,z)$ is a monic polynomial \eqref{eq:basicsplit1} in $\IC\llb w,y\rrb[z]$ which splits in $\IC(w^{1/p})\llb y-y_0\rrb[z]$.
Then $f = \prod_{i=1}^l f_i$, where each $f_i \in \IC(w)\llb y-y_0\rrb[z]$ is an irreducible monic polynomial in $z$ of degree $k_i$,
which splits in $\IC(w^{1/k_i})\llb y-y_0\rrb[z]$,
and $k_1 +\cdots + k_l = k$.
\end{enumerate}
\end{lemma}

\begin{proof}
(1) is an immediate consequence of Lemma \ref{lemma:finitesplit}. 
For (2), we can again assume that $y_0 = 0$. Write
$$
f(v^{p},y,z) = \prod_{j=1}^k \left(z-b_j(v,y)\right),
$$
where each $b_j \in \IC(v)\llb y\rrb$. The $p$th roots of unity $e^{2\pi il/p}$, $l=1,\ldots,p$, have the structure of a
cyclic group $\IZ_{p}$. Let $\ep = e^{2\pi i/{p}}$. Then the ordered set $\{b_j(\ep v,y)\}$ is a permutation of the set of 
roots $\{b_j(v,y)\}$; say,
$b_j(\ep v,y) = b_{s(j)}(v,y)$. Then $b_1(\ep^2 v,y) = b_{s(1)} (\ep v,y) = b_{s^2(1)} (v,y)$, and $b_1(\ep^l v,y) = b_{s^l(1)}(v,y)$,
for all $l$. So there is a homomorphism of $\IZ_{p}$ onto a cyclic subgroup $\IZ_m$ of the group of permutations of the roots $b_j$,
for some $m\leq k$.

Then, after reordering, $\prod_{j=1}^m \left(z-b_j(v,y)\right)$ is
invariant under the action of $\IZ_{p}$ and, therefore, an element of $\IC(v^{p})\llb y\rrb[z]$. If $m<k$, this means
that $f(w,y,z)$ is not irreducible in $\IC(w)\llb y\rrb[z]$.

We can assume that $f(w,y,z)$ is irreducible in $\IC(w)\llb y\rrb[z]$. Then $m=k$. 
In general, the group of homomorphisms
$\IZ_p \to \IZ_m$  is isomorphic to $\IZ_d$, where $d = \gcd(p,m)$.
Any $h\in\IZ_d$ corresponds to the homomorphism $\IZ_p \to \IZ_m$ given by $\ep \mapsto \ep^{hp/d}$.
If $f$ is irreducible, then $m=k$ and $\IZ_p \to \IZ_k$ is onto, so that $d=k$ and $f(v^k,y,z)$ splits as required.
\end{proof}

\begin{remark}\label{rem:finitesplit}
Likewise, if $w$ is a single variable and $f(w,y,z)$ is a monic polynomial \eqref{eq:basicsplit1} in $\IC\llb w,y\rrb[z]$ which splits in 
$\IC\llb w^{1/{p}},y\rrb[z]$,
then $f = \prod_{i=1}^l f_i$, where each $f_i \in \IC\llb w, y\rrb[z]$ is an irreducible monic polynomial in $z$ of degree $k_i$,
which splits in $\IC\llb w^{1/k_i}, y\rrb[z]$, and $k_1 +\cdots + k_l = k$.

Note that the polynomial $f(w,x,z) = z^2 + (w^2 + x)x^2$ is irreducible in $\IC\llb w,x\rrb[z]$ but not in
$\IC(w)\llb x\rrb[z]$.
\end{remark}

\subsection{Generic normal crossings and the discriminant}\label{subsec:disc}
Let $X \hookrightarrow Z$ denote an embedded hypersurface ($Z$ smooth).
For any $k\in \IN$, $\{x\in X: X \text{ is } \nc(k) \text{ at } x\}$ is a smooth subspace of $X$ of codimension
$k$ in $Z$.

\begin{lemma}\label{lemma:genericnc}
The set of non-normal crossings points of $X$ is a closed algebraic (or analytic) subset.
If $Y$ is an irreducible subset of $X$ and $X$ is generically $\nc(k)$ on $Y$, for some $k\in \IN$, then
$\{x\in Y: X \text{ is not } \nc(k) \text{ at } x\}$ is a proper closed algebraic (or analytic) subset of $Y$.
\end{lemma}

\begin{proof}
This is a simple consequence of the following two facts. (1) The desingularization invariant $\inv = \inv_{X}$ 
(in year zero)
is Zariski upper-semicontinuous on $X$. (2) $X$ is $\nc(k)$ at a point $a$ if
and only if $\inv_X(a) = (k,0,1,0,\ldots,1,0,\infty)$ (with $k$ pairs) and $X$ has $k$ local analytic branches at $a$
(equivalently, there are precisely $k$ points in the fibre of the normalization of $X$ over $a$; see \cite[Thm.\,3.4]{BDMV}).
\end{proof}

Lemmas \ref{lemma:disc} and \ref{lemma:Dsq} following deal with the question of splitting in terms of the
discriminant. These results will be used in Section \ref{sec:triplenc}.
Lemma \ref{lemma:Dsq} in the case $k=3$ was proved in \cite[Lemmas 3.4,\,3.5]{BLMmin2},
but the general proof below is much simpler. 

Let $f$ denote a regular function, written in
\'etale local coordinates
\begin{equation*}
(w,u,x,z) = (w_1,\ldots,w_r, u_1,\ldots,u_q, x_1,\ldots,x_{k-1},z)
\end{equation*}
as
\begin{equation}\label{eq:weier1}
f(w,u,x,z) = z^k + a_1(w,u,x)z^{k-1} +\cdots + a_k(w,u,x).
\end{equation}
Let $D(w,u,x)$ denote the discriminant of $f(w,u,x,z)$ as a polynomial in $z$. The discriminant $D$ is
a weighted homogeneous polynomial of degree $k(k-1)$ in the coefficients $a_i$, 
where each $a_i$ has weight $i$.

\begin{lemma}\label{lemma:disc}
Assume that $f$ is in the ideal generated by $x_1,\ldots,x_{k-1},z$, and that $f$ splits formally (into $k$ factors of order $1$)
at every point where $x=z=0$ and $w_1\cdots w_r \neq 0$.
Then $D$ factors in an \'etale neighbourhood of $a=0$ as
\begin{equation}\label{eq:D}
D = \Phi^2 \Psi,
\end{equation}
where $\Phi$ is in the ideal generated by $x_1,\ldots,x_{k-1}$, and
$\Psi$ is nonvanishing outside $\{w_1\cdots w_r = 0\}$.
\end{lemma}

\begin{proof}
The hypotheses imply that $D$ is a square (\'etale locally) at every point where $x=z=0$
and $w_1\cdots w_r \neq 0$.
So there is an \'etale neighbourhood of $a$ in which every irreducible factor of $D$ occurs to even 
power, except for those factors which are nonvanishing outside $\{w_1\cdots w_r = 0\}$.
\end{proof}

\begin{lemma}\label{lemma:Dsq}
Assume that $f$ satisfies the hypotheses of Lemma \ref{lemma:disc} and that $D$ factors in an \'etale 
coordinate neighbourhood of $a=0$ as in \eqref{eq:D}. Then,
after a finite number of blowings-up with centres of the form $\{z=x=w_j=0\}$, for some $j$, we can assume
that $D(v_1^2,\ldots, v_r^2, u, x)$ is a square.
\end{lemma}

\begin{proof}
We can assume that $a_1 = 0$ in \eqref{eq:weier1} (by completing the $k$th power).
According to Lemma \ref{lemma:disc},
$$
\Psi(u,w,x) = \xi(u,w) +  x_1\theta_1(u,w,x) +\cdots + x_{k-1}\theta_{k-1}(u,w,x),
$$
where $\xi(u,w)$ does not vanish outside $\{w_1\cdots w_r = 0\}$; i.e., the zero set of $\xi$ is a
subset of $\{w_1\cdots w_r = 0\}$, so that $\xi(u,w) = w^\al \eta(u,w)$, where 
$w^\al = w_1^{\al_1}\cdots w_r^{\al_r}$
is a monomial and $\eta(u,w)$ is a unit. If $\al = 0$, then $D$ is already a square, so we can
assume that $\al\neq 0$.

Consider the blowing-up $\s$ with centre $\{z = x = w_j = 0\}$, for some $j$ such that
$\al_j \neq 0$. The subspace $\{z=x=0\}$ lifts to the $w_j$-chart of $\s$, given by substituting 
$(w,u,w_j x,w_j z)$ for $(w,u,x,z)$, and we have
\begin{align*}
f'(w,u,x,z) &:= w_j^{-k} f(w,u,w_jx,w_jz)\\
            &= z^k + a'_2(w,u,x)z^{k-2} + \cdots a'_k(w,u,x),  
\end{align*}
where each $a'_i(w,u,x) = w_j^{-i} a_i(w,u,w_jx)$. Since $D \in (x)^{k(k-1)}$, 
$f'(w,u,x,z)$ has discriminant
\begin{equation*}
D' = w_j^{-k(k-1)} D\circ\s = (\Phi')^2\cdot \Psi\circ\s,
\end{equation*}
and
\begin{equation*}
(\Psi\circ\s)(w,u,x) = w_j\left(w^{\al'}\eta'(w) + \theta'(w,x)\right),
\end{equation*}
where $w^{\al'} = w_1^{\al_1}\cdots w_j^{\al_j-1}\cdots w_r^{\al_r}$, 
$\eta'$ is a unit and $\theta' \in (x)$.

It follows that, after $\al_1 + \cdots + \al_r$ blowings-up with centres of the form
$\{z=x=w_j = 0\}$, for some $j$, $D(v_1^2,\ldots, v_r^2, u,x)$ is a square.
\end{proof}

\begin{remark}\label{rem:Dsq}
In Section \ref{sec:alg}, we will deal with an embedded hypersurface $X\hookrightarrow Z$ together with
a simple normal crossings divisor $E$, and will need to apply Lemma \ref{lemma:Dsq} and Theorem \ref{thm:splitintro}
(proved in \S\ref{subsec:SV} following) to a function $g(y_1,\ldots,y_r, w,u,x,z) = y_1\cdots y_r f(y,w,u,x,z)$, where the $y_j$
are local generators of the components of $E$, $f$ is as in \eqref{eq:weier1} with coefficients $a_i = a_i(y,w,u,x)$, and
$f$ satisfies the hypotheses of Lemma \ref{lemma:Dsq}. The latter holds in this case with centres $\{z=x=y=w_j=0\}$, 
and the proof is the same.
\end{remark}

\subsection{Limits of $\nc(k)$ in $k+1$ variables}\label{subsec:SV}
In this subsection, we prove Theorem \ref{thm:splitintro}.
See also Lemma \ref{cor:finitesplit} and Remark \ref{rem:finitesplit}. The statement
of Theorem \ref{thm:splitintro} means, more precisely, 
that, after finitely many blowings-up of $0$, the strict transform of $f$
splits at the inverse image of $0$ in the lifting of the $w$-axis $\{x=z=0\}$. Of course, after blowing up
$0$, the $w$-axis lifts to the $w$-axis in coordinates of the $w$-chart, given by $(w,wx,wz)$.

\begin{proof}[Proof of Theorem \ref{thm:splitintro}]
Let us change notation and write $x = (x_1,\ldots, x_{k-1}, x_k)$, where $x_k = w$. Given any 
field $K$, we write $K\llbr t^{1/q}\rrbr$ to denote the field of Puiseux Laurent series in $t^{1/q}$, 
where $q$ is a positive integer.

Let $\SLplus(k,\IZ)$ denote the multiplicative subsemigroup of $\SL(k,\IZ)$ consisting of 
upper-triangular matrices
$$
A = \left(\begin{matrix}
1 & a_{12} & \cdots & a_{1k}\\
0 & 1         & \cdots & a_{2k}\\
\vdots  & \vdots  & \ddots & \vdots\\
0 & 0         & \cdots & 1
\end{matrix}\right)\,,
$$
where the $a_{ij}$ are nonnegative integers. Clearly, $\SLplus(k,\IZ)$ acts on monomials
$x^\al = x_1^{\al_1}\cdots x_k^{\al_k}$ by $x^\al \mapsto x^{\al A}$, $A \in \SLplus(k,\IZ)$, where
$$
\al A := (\al_1,\ldots,\al_k)\cdot\left(\begin{matrix}
1 & a_{12} & \cdots & a_{1k}\\
0 & 1         & \cdots & a_{2k}\\
\vdots  & \vdots  & \ddots & \vdots\\
0 & 0         & \cdots & 1
\end{matrix}\right)\,.
$$
Write $\psi_A(x^\al) := x^{\al A}$. Of course, $\psi_A$ extends to an operation on
$\IC\llb x \rrb = \IC\llb x_1,\ldots, x_k\rrb$, and to an operation on $\IC\llb x\rrb [z]$ (by the
preceding operation on coefficients), which we also denote $\psi_A$, in each case.

Since $\psi_A$ takes $x_k = w \mapsto w$ and (for each $i=1,\ldots,k-1$) takes $x_i \mapsto x_i$
times a monomial in $(x_{i+1},\ldots,x_{k-1}, w)$ (the monomial with exponents given by the $i$th row
of $A$), we see that $\psi_A$ also makes sense as an operation on $\overline{\IC(w)}\llb x_1,\ldots,x_{k-1}\rrb$,
or on $\overline{\IC\llbr w\rrbr}\llb x_1,\ldots,x_{k-1}\rrb$.

By the theorem of Soto and Vicente \cite{SV}, there exists a positive integer $q$ such that $f$ splits
in $\IC\llbr x_k^{1/q}\rrbr \cdots \llbr x_1^{1/q}\rrbr [z]$ and, moreover, there exists $A \in \SLplus(k,\IZ)$ such
that $\psi_A(f)$ splits in $\IC\llb x_1^{1/q},\ldots, x_k^{1/q}\rrb [z]$. Let $c_i \in \IC\llb x_1^{1/q},\ldots, x_{k-1}^{1/q}, w^{1/q}\rrb$,
$i=1,\ldots,k$, denote the roots of $\psi_A(f)$.

By Lemma \ref{lemma:basicsplit}, $f$ splits in $\overline{\IC(w)}\llb x_1,\ldots, x_{k-1}\rrb [z]$. Let $b_i \in \overline{\IC(w)}\llb x_1,\ldots, x_{k-1}\rrb
\subset \overline{\IC\llbr w\rrbr}\llb x_1,\ldots, x_{k-1}\rrb$, $i=1,\ldots,k$, denote the roots of $f$.

By the uniqueness of formal expansion, the set $\{c_i\}$ of roots of $\psi_A(f)$ coincides with the set $\{\psi_A(b_i)\}$;
i.e., each $c_i \in \IC \llb x_1,\ldots, x_{k-1}, w^{1/q}\rrb$.

Note that, given any monomial $x_1^{\al_1}\cdots x_{k-1}^{\al_{k-1}}$, $w=x_k$ appears in 
$\psi_A(x_1^{\al_1}\cdots x_{k-1}^{\al_{k-1}})$
to the power $\sum_{j=1}^{k-1} \al_j a_{jk}$; i.e., to a power at most $d\mu$, where $d$ is the degree $\al_1 +\cdots +\al_{k-1}$ and
$\mu = \max\{a_{jk}\}$.

It follows that blowing up the origin $\mu$ times will clear all denominators in the roots $b_i$.
\end{proof}

\subsection{Normality}\label{subsec:norm}
The purpose of this subsection is to give an elementary proof that circulant singularities have
smooth normalization (Proposition \ref{prop:circ}), and that, 
in Theorem \ref{thm:splitintro}, we get smooth 
normalization after finitely many blowings-up of the origin (see Corollary \ref{cor:lairez}).
As usual, $\IC$ can be replaced by any algebraically closed field in the results following.

\begin{proposition}\label{prop:lairez}
Let $f \in \IC\llb w,x\rrb[z] = \IC\llb w,x_1,\ldots,x_n\rrb[z]$ denote an irreducible monic polynomial
of degree $k$ in $z$ with coefficients in $\IC\llb w,x\rrb$; i.e.,
\begin{equation*}
f(w,x,z) = z^k + a_1(w,x)z^{k-1} + \cdots + a_k(w,x),
\end{equation*}
where each $a_i(w,x)\in \IC\llb w,y\rrb$. Let $R := \IC\llb w,x\rrb[z]/(f)$, and let $R'$ denote the integral closure of $R$ in its field of fractions $\Frac(R)$.
Then the following are equivalent:
\begin{enumerate}
\item[(1)] $f(w,x,z)$ splits into $k$ factors in $\IC\llb w^{1/k}, x\rrb[z]$;
\item[(2)] There exists $u\in R'$ such that $u^k=w$.
\end{enumerate}
\end{proposition}

\begin{proof}[Proof (due to Pierre Lairez)]
Let $A := \IC\llb w,x\rrb$ and let $K$ denote the field of fractions of $A$. The polynomial $p(w,y) := y^k - w$
is irreducible in $A[y]$ or in $K[y]$, by Eisenstein's criterion. In particular, $(p)$ is a maximal ideal in $K[y]$, so
that $K[w^{1/k}] = K[y]/(p)$ is a field, and, therefore, $K[w^{1/k}]$ is the field of fractions of $A[w^{1/k}] = A[y]/(p)$.

The ring $A[w^{1/k}]$ is a unique factorization domain, so that $f$ splits over $A[w^{1/k}]$ (i.e., $f$ splits in $A[w^{1/k}][z]$) 
if and only if $f$ splits over its field of fractions $K[w^{1/k}]$.

By hypothesis, $f$ is irreducible in $A[z]$, and therefore in $K[z]$. 
By Lemma \ref{lem:lairez} following, $f$ splits over the field $K[y]/(p) = K[w^{1/k}]$
if and only if $p$ splits over $K[z]/(f)$. Now, $K[z]/(f)$ is isomorphic to the field of fractions $\Frac[R]$ of $R$, and, of course, 
$p(w,y)$ splits over $\Frac(R)$ if and only if $p(w,y)$ has a root in $\Frac(R)$; since $p(w,y)$ is monic, such a root would belong to $R'$. 
\end{proof}

\begin{remark}\label{rem:lairez}
We recall that, if $f(x)$ is an irreducible polynomial (of a single variable $x$) with coefficients in a field $K$,
then the ideal $(f(x)) \subset K[x]$ is maximal, so that $L_f := K[x]/(f(x))$ is field ($L_f$ is the splitting field of $f(x)$
over $K$). In general, if $A$ is a reduced Noetherian ring, let $Q(A)$ denote the \emph{total
quotient algebra} $Q(A) := S^{-1}A$, where $S$ is the multiplicative subset of $A$ consisting of non-zerodivisors.
Then $Q(A)$ is the product of the fields of fractions $Q(A/\frak{p_i})$, where $\frak{p}_1,\ldots,\frak{p}_r$ are
the minimal primes of $A$. In particular, if $p(x) \in L[x]$ is a polynomial with coefficients in a field $L$, and
with $r$ distinct irreducible factors, then $Q(L[x]/(p(x))$ is a product of $r$ fields that are uniquely determined by $p(x)$.
\end{remark}

\begin{lemma}\label{lem:lairez}
Let $f(x) \in K[x]$ and $g(y) \in K[y]$ both denote irreducible polynomials of a single variable over a field $K$.
Then the number of irreducible factors of $f(x)$ over the field $L_g = K[y]/(g(y))$ equals the number of irreducible
factors of $g(y)$ over the field $L_f = K[x]/(f(x))$. 
\end{lemma}

\begin{proof}
The algebra $K[x,y]/(f(x),g(y))$ can be identified with both $L_f[y]/(g(y))$ and $L_g[x]/(f(x))$, so the assertion
is an immediate consequence of Remark \ref{rem:lairez}.

As an alternative argument, we can use that fact that the irreducible factors of $f(x)$ over $L_g$
are in one-to-one correspondence
with the irreducible components of $\Spec (L_g[x]/(f(x))) = \Spec (K[x,y]/(f(x),g(y)))$, and likewise for the
irreducible factors of $g(y)$.
\end{proof}

\begin{corollary}\label{cor:lairez}
With the hypotheses of Proposition \ref{prop:lairez},
if either of the (equivalent) conditions (1),\,(2) of the proposition holds,
then $R'$ is regular.
\end{corollary}

\begin{proof}
By condition (1) of Proposition \ref{prop:lairez}, we can write
$$
f(y^k,x,z) = \prod_{i=0}^{k-1} g(\ep^i y,x,z),
$$
where $g \in \IC\llb y,x\rrb[z]$ and $\ep = e^{2\pi i/k}$. Of course,
\begin{equation}\label{eq:div}
\prod_{i=0}^{k-1} g(\ep^i y,x,z) - f(w,x,z) = (y^k - w) h(y^k,w,x,z),
\end{equation}
where $h \in \IC\llb y,w,x\rrb[z]$.

By condition (2), $y^k-w$ has a root $u$ in $R' \subset \Frac(R)$, and the homomorphism from
$R[y]$ onto $R[u] \subset R'$ induced by $y\mapsto u$ has kernel $(g(\ep^iy,x,z), y^k - w)$, for some $i=0,\ldots,k-1$, by \eqref{eq:div}.
Therefore,
$$
R[u] \cong \frac{R[y]}{(g(\ep^iy,x,z), y^k-w)} \cong \frac{\IC\llb y,w,x\rrb[z]}{(g(\ep^iy,x,z),\,y^k-w)} 
\cong \frac{\IC\llb y,x\rrb[z]}{(g(\ep^iy,x,z))} \cong  \IC\llb y,x\rrb.
$$
In particular, $R[u] \subset R'$ is regular and hence already integrally closed, and, therefore, 
coincides with $R'$, as required.
\end{proof}

\begin{example}\label{ex:lairez}
The variety $X := \{z^2 -w_1w_2 x^2\}$ has singular normalization $\{z^2 - w_1w_2 = 0\}$. Let $w_1=y_1^2,\,w_2 = y_2^2$. Then
$z^2 -w_1w_2 x^2$ splits as $(z-y_1y_2x)(z+y_1y_2x)$. The mapping to $X$ from each irreducible component $\{z \pm y_1y_2x=0\}$
is generically $2$-to-$1$.
\end{example}

\smallskip
\section{Limits of $k$-fold normal crossings in $k+1$ variables}\label{sec:lim}
In this section, we prove Theorem \ref{thm:limintro}. The proof consists of two parts.
The first part is formulated as Theorem \ref{thm:lim} below. Theorem \ref{thm:lim} begins
with the hypotheses of Theorem \ref{thm:splitintro}, and the proof provides a construction
in \'etale (or analytic) local coordinates that proves the assertion of Theorem \ref{thm:limintro}, although it may
not be evident \emph{a priori} that the blowings-up involved are global admissible smooth
blowings-up.

There are actually two sequences of blowings-up involved in the proof of Theorem \ref{thm:limintro}
or Theorem \ref{thm:lim}. First there is a sequence of \emph{cleaning blow-ups}, following \cite[Section 2]{BMmin1}, after which $X$ can
be described locally at a limit of $\nc(k)$ points, by a certain \emph{pre-circulant normal form}. In the
irreducible case, for example, the latter means an equation of the following form in suitable local
\'etale (or analytic) coordinates:
\begin{equation}\label{eq:precirc}
\De_k\left(z, w^{n_1 + 1/k}x_1,\ldots, w^{n_{k-1} + (k-1)/k} x_{k-1}\right) = 0
\end{equation}
(where we can be more precise about the integers $n_j$---see Remark \ref{rem:minprecirc}).

Once we have such local coordinates, there is a second blowing-up sequence which we
use to reduce each $n_j$ to zero, to get circulant normal form. For example, given $j$, we
can reduce $n_j$ to zero by blowings-up with centre $\{z = w = 0,\,  x_\ell = 0, \text{ for all } \ell\neq j\}$.
We will show in a simple fashion how to choose local coordinates for \eqref{eq:precirc} so that
these centres of blowing up make sense in a global combinatorial way. A similar idea will be used
in the proof of Proposition \ref{prop:lim3} as well as in \S\ref{subsec:overviewproof}\,(B)(II) and \S\ref{subsec:cp4}.

The purpose of the second part of the proof of Theorem \ref{thm:limintro} is to describe
the first sequence of blowings-up above in an
invariant global way that is independent of the local construction in the proof of Theorem
\ref{thm:lim}; these blowings-up are called \emph{cleaning blow-ups}, following 
\cite[Section 2]{BMmin1}. Similar ideas have been developed by Koll\'ar \cite{Kol2} and
Abramovich, Temkin and W{\l}odarczyk \cite{ATW}.

The proof of Theorem \ref{thm:lim} requires little knowledge of the technical details of the desingularization
algorithm, except for a very basic understanding of maximal contact and the coefficient
ideal. A reader unfamiliar with the technology of desingularization can safely read
the rest of the article without the details of the second part of the proof of 
Theorem \ref{thm:limintro}. At the same time, the proof
of Theorem \ref{thm:lim} introduces some of this technology by explicit local
computation that we hope may be helpful in understanding the cleaning procedure, as
described in \S\ref{subsec:clean}. Some basic details of the desingularization
algorithm and the invariant $\inv$ are recalled in \S\ref{subsec:inv} and will be needed
also in Section \ref{sec:alg}. The reader is again referred to \cite{BMfunct} and the \emph{Crash
course on the desingularization invariant} \cite[Appendix]{BMmin1} for all of the notions
from resolution of singularities that we use.

\subsection{Circulant normal form}\label{subsec:lim}

\begin{theorem}\label{thm:lim}
Assume that (after an $\inv$-admissible sequence of blowings-up); cf. \S\ref{subsec:inv} below)
$f(w,x_1,\ldots,x_{k-1},z)$ satisfies the hypotheses of Theorem \ref{thm:splitintro}, and that
$$
\inv(0) = \inv(\nc(k)).
$$
Then there is a finite sequence of admissible blowings-up(in fact, admissible for the
truncated invariant $\inv_1$; see \S\ref{subsec:inv}) that are isomorphisms 
over the $\text{nc}(k)$ locus, after
which the only singularities that may occur as limits of $\text{nc}(k)$ points, are products of circulant
singularities.

More precisely, assume that $f = f_1\cdots f_s$, where each $f_i$ is an irreducible polynomial
\begin{equation*}
f_i(w,x,z) = z^{k_i} + \sum_{j=1}^{k_i} a_{ij}(w,x) z^{k_i - j}
\end{equation*}
with regular or algebraic (or analytic) coefficients,
and $k_1 + \cdots + k_s = k$. Then there is a finite sequence of admissible blowings-up that are isomorphisms over 
the $\text{nc}(k)$ locus, after which, at the limit of the $\text{nc}(k)$ points, there is a local \'etale 
(or analytic) coordinate system 
$\left(w, (y_{i\ell})_{\ell=0,\ldots,k_i -1,\, i=1,\ldots s}\right)$
in which the strict transform of $f(w,x,z)=0$ is given by
\begin{equation}\label{eq:circnormal}
\prod_{i=1}^s \De_{k_i} \left(y_{i0}, w^{1/k_i} y_{i1}, \ldots, w^{(k_i-1)/k_i}y_{i,k_i-1}\right) = 0.
\end{equation}
\end{theorem}

\begin{remark}\label{rem:limE}
In Section \ref{sec:alg}, where we treat $X$ together with a simple normal crossings divisor $E$,
we will apply Theorem \ref{thm:lim} to a product of local generators of the ideals of $X$ and the
components of $E$; see also Remark \ref{rem:Dsq}. The circulant normal form \eqref{eq:circnormal}
then becomes
\begin{equation*}
y_1\cdots y_r \prod_{i=1}^s \De_{k_i} \left(y_{i0}, w^{1/k_i} y_{i1}, \ldots, w^{(k_i-1)/k_i}y_{i,k_i-1}\right) = 0,
\end{equation*}
where the $\{y_j=0\}$ are the strict transforms of the components of $E$.
\end{remark}

\begin{proof}[Proof of Theorem \ref{thm:lim}]
By Theorem \ref{thm:splitintro}, after a finite number of blowings-up of the origin, we can assume
that $f$ splits in $\IC\llb w^{1/p},x\rrb[z]$, for some positive integer $p$.

We will first consider the case that $f$ is irreducible (so we get $\text{cp}(k)$ as
limit), and then handle the general case.

\medskip\noindent
\emph{Irreducible case.} 
By Lemma \ref{cor:finitesplit}(2) and Remark \ref{rem:finitesplit}, we can take $p=k$; i.e.,
$f$ splits in $\IC\llb w^{1/k},x\rrb[z]$, and we can write
\begin{equation*}
f(v^k,x,z) = \prod_{\ell=0}^{k-1} (z + b(\ep^\ell v, x)),
\end{equation*}
where $b(v,x) \in \IC\llb v,x\rrb$ and $\ep = e^{2\pi i/k}$. 

We can assume that $a_1(w,x)=0$ (by the Tschirnhausen transformation; i.e., completion of the
$k$th power). Set 
\begin{equation*}
Y_\ell := z + b(\ep^\ell v, x), \quad \ell = 0,\ldots,k-1,
\end{equation*}

\medskip\noindent
and define $X_0,\ldots,X_{k-1}$ by \eqref{eq:inveigenmatrix}; i.e.,
\begin{align}\label{eq:inverse1}
X_0 &= \frac{1}{k}\sum_{j=0}^{k-1} Y_j = z,\nonumber \\
X_\ell &= \frac{1}{k} \sum_{j=0}^{k-1} \ep^{\ell(k-j)} Y_j =  \frac{1}{k} \sum_{j=0}^{k-1} \ep^{\ell(k-j)} b(\ep^j v,x),\quad \ell =1,\ldots,k-1.
\end{align}

It is easy to check that, for each $\ell =1,\ldots,k-1$, $v^{k-\ell} X_\ell$ is invariant under the action of the group $\IZ_k$ of
$k$th roots of unity (where the operation of $\ep$ on functions of $(v,x)$ is induced by $(v,x) \mapsto (\ep v,x)$).
In other words,
\begin{equation*}
v^{k-\ell} X_\ell = \eta_\ell (v^k, x),
\end{equation*}
where $\eta_\ell(w,x) \in \IC\llb w,x\rrb$, $\ell =1,\ldots,k-1$. Since each $\eta_\ell$ must, therefore, be divisible by $v^k$,
we can write
\begin{equation*}
X_\ell = v^{k m_\ell + \ell} \zeta_\ell (v^k, x) = w^{m_\ell + \ell/k} \zeta_\ell(w,x),
\end{equation*}
where $m_\ell$ is a nonnegative integer and $\zeta_\ell(w,x) \in \IC\llb w,x\rrb$ is not divisible by $w$, $\ell =1,\ldots,k-1$.

Since the $X_1,\ldots,X_{k-1}$ are given by an invertible linear combination of $b(\ep v, x),\allowbreak\ldots,b(\ep^{k-1}v,x)$
(recall that $\sum_{\ell=0}^{k-1} b(\ep^\ell v,x) = 0$) , it
follows that the \emph{coefficient ideal} of the marked ideal $(f,k)$ is equivalent to the marked ideal
\begin{equation*}
\ucC^1 := \left( \left(w^{km_\ell + \ell}\zeta_\ell^k\right), k\right)
\end{equation*}
on the \emph{maximal contact subspace} $N^1 := \{z=0\}$ (cf. \cite[Example A.13]{BMmin1}). Set 
\begin{equation}\label{eq:mon1}
\al_1 := \min_{1\leq \ell\leq k-1} \left(k m_\ell + \ell\right),
\end{equation}
and let $\ell_1$ denote the (unique) corresponding $\ell$ (realizing the minimum), and $\xi_{\ell_1}$ the corresponding $\zeta_\ell$. 
Then $\xi_{\ell_1}$ has order $1$, since $\inv(a) = (k,0,1,\dots)$. The monomial $w^{\al_1}$ generates
the monomial part of the coefficient ideal $\ucC^1$ (cf. \cite[\S A.6]{BMmin1}). Set $p_1 := m_{\ell_1}$.

It follows that the second coefficient ideal $\ucC^2$ (still with marked or associated order $k$), on the second
maximal contact subspace $N^2 := \{z=\xi_{\ell_1}=0\}$, is generated by
\begin{equation*}
w^{km_{\ell} + \ell -\al_1} \zeta_{\ell}^k |_{z=\xi_{\ell_1}=0}\,,\quad \ell\neq \ell_1.
\end{equation*}
Therefore, for each $\ell\neq\ell_1$, there is a nonnegative integer $\widetilde{m}_\ell$
such that
\begin{equation*}
\zeta_\ell = \eta_\ell^1 + w^{\widetilde{m}_\ell}  \xi_\ell,
\end{equation*}
where $\eta_\ell^1$ is in the ideal generated by $\xi_{\ell_1}$, and $\xi_\ell|_{z=\xi_{\ell_1}=0}$ is not divisible by $w$. Let
\begin{equation*}
\al_2 := \min_{\ell\neq \ell_1} \left(k (m_\ell + \widetilde{m}_\ell)  + \ell - \al_1\right),
\end{equation*}
and let $\ell_2 \neq \ell_1$ denote the (unique) corresponding $\ell$. 
Then $\xi_{\ell_2}$ has order $1$, since $\inv(a) = (k,0,1,0,1,\dots)$. The monomial $w^{\al_2}$ generates
the monomial part of the coefficient ideal $\ucC^2$. Clearly, $\al_2 = kp_2 + \widetilde{h}_2$,
where $p_2,\,\widetilde{h}_2$ are nonnegative integers and $1\leq \widetilde{h}_2 < k$.

We repeat this construction for each $j=3,\ldots,k-1$. 

We now apply cleaning blow-ups, with centre $N^j \,\cap\, \{w=0\}$, to successively reduce each $p_j$ to zero, $j=k-1, k-2,\ldots, 1$. In particular, for each $j$, we reduce $\widetilde{m}_{\ell_j}$ to $0$ since
$\eta_{\ell_j}^{j-1}$ is in the ideal generated by $\xi_{\ell_1},\ldots,\xi_{\ell_{j-1}}$.

We can then make a formal (or \'etale) coordinate change
\begin{align*}
y_{\ell_1} &:= \xi_{\ell_1},\\
y_{\ell_j} &:= \eta_{\ell_j}^{j-1} +  \xi_{\ell_j}, \quad j=2,\dots,k-1,
\end{align*}
to reduce each $X_\ell$ to $w^{n_\ell + h_\ell/k} y_\ell$,\, $\ell =1,\ldots,k-1$, where $\{h_1,\ldots,h_{k-1}\} = \{1,\ldots,k-1\}$,
each $n_\ell$ is a nonnegative integer, and $n_{\ell_1} = 0$. (After re-ordering indices if necessary, and relabelling variables,
this means that we have reduced the equation $f(w,x,z)=0$ 
to a pre-circulant singularity \eqref{eq:precirc}.

\begin{remark}\label{rem:minprecirc}
The cleaning computation above shows that \eqref{eq:precirc} can be
written in the following way with minimal choice of $n_\ell$:
\begin{equation*}\label{eq:minprecirc}
\De_k\left(z, w^{h_1/k}x_1, w^{n_2 + h_2/k}x_2\ldots, w^{n_{k-1} + h_{k-1}/k} x_{k-1}\right) = 0,
\end{equation*}
where $\{h_1,\ldots,h_{k-1}\} = \{1,\ldots,k-1\}$ and each $n_\ell = \min\{m\in \IN: m + h_\ell/k > n_{\ell-1} + h_{\ell-1}/k\}$,
$n_1 = 0$.
\end{remark}

\begin{example}\label{ex:minprecirc}
The equation
\begin{equation*}
\De_3(z, w^{2/3}x_1, w^{4/3}x_2) = 0
\end{equation*}
is in (minimal) pre-circulant normal form, but not in circulant normal form.
\end{example}

Now, beginning with pre-circulant normal form \eqref{eq:precirc}, we can blow up to reduce each $n_\ell$ to $0$,
as described in the introduction to this section. For later reference, we write the argument as the following
remark.

\begin{remark}\label{rem:D_1trick}
To make the blowings-up described in a global way, we take advantage
of the normal form \eqref{eq:precirc} to introduce a small trick or \emph{astuce} that will be repeated
in the proof of Proposition \ref{prop:lim3}, as well as in \S\ref{subsec:overviewproof}\,(B)(II) and \S\ref{subsec:cp4}. 
We first make a single blowing-up $\s$ with
centre $\{0\}$. This blowing-up
does not change \eqref{eq:precirc}, which is transformed to the same equation in the $w$-chart of the blowing-up $\s$,
given by substituting $(w, wx, wz)$ for the original variables $(w,x,z)$.

But $\{w=0\}$ is now the exceptional divisor $D_1$ of $\s$, and, in the new coordinates $(w,x,z)$
(after the substitution above), the centre of blowing up $\{z = w = 0,\, x_j =0, \text{ for all } j \text{ where } n_j=0\}$, needed
to decrease all positive $n_j$,
for example, extends to a global smooth subvariety of $D_1$; more precisely, the blowing-up in the $w$-chart extends to a global
admissible blowing-up which can be described in an explicit way in every coordinate chart of $\s$. We can continue, to decrease each $n_j$ to $0$.
\end{remark}

\begin{remark}\label{rem:tschirn}
It is not necessary to assume that $a_1 = 0$ in the proof for the irreducible case above (i.e., in the hypotheses of Theorem \ref{thm:splitintro}); 
the Tschirnhausen transformation will appear
naturally in the construction (see \eqref{eq:tschirn1} below). This is, in fact, more convenient in the general case following.
\end{remark}

\medskip\noindent
\emph{General case.} Consider $f=f_1\cdots f_s$ as in the statement of the theorem. 
By Lemma \ref{cor:finitesplit}(2) and Remark \ref{rem:finitesplit}, for each $i=1,\ldots,s$, we can write
\begin{equation*}
f_i(v^{k_i},x,z) = \prod_{\ell=0}^{k_i -1} (z + b_i(\ep_i^\ell v, x)),
\end{equation*}
where $b_i(v,x) \in \IC\llb v,x\rrb$ and $\ep_i = e^{2\pi i/k_i}$. 

For each $i=1,\ldots,s$ and $\ell = 0,\ldots,k_i -1$, set
\begin{equation*}
Y_{i\ell} := z + b_i(\ep_i^\ell v,x),
\end{equation*} 
and define $X_{i\ell}$ as in \eqref{eq:inveigenmatrix}; i.e.,
\begin{equation*}
X_{i\ell} = \frac{1}{k_i} \sum_{j=0}^{k_i-1} \ep_i^{\ell(k_i -j)} Y_{ij}.
\end{equation*}
Then, for each $i$,
\begin{align*}
X_{i0} &= z + \frac{1}{k_i} a_{i1}(v^{k_i},x),\\
X_{i\ell} &= \frac{1}{k_i} \sum_{j=0}^{k_i -1} \ep_i^{\ell(k_i -j)} b_i(\ep_i^j v,x),\quad \ell =1,\ldots,k_i-1;
\end{align*}
hence,
\begin{align}\label{eq:tschirn1}
X_{i0}(w,x,z) &= z + \frac{1}{k_i} a_{i1}(w,x),\\
X_{i\ell}(w,x) &= w^{m_{i\ell} + \ell/k_i} \zeta_{i\ell}(w,x), \quad \ell =1,\ldots,k_i-1,\nonumber
\end{align}
where each $m_{i\ell}$ is a nonnegative integer and each $\zeta_{i\ell}(w,x)\in \IC\llb w,x\rrb$ is
not divisible by $w$.

We can use $\{X_{i0}(w,x,z) = 0\}$, for any choice of $i$, as the first maximal contact subspace $N^1$; let us take
\begin{equation*}
N^1 := \{X_{10} = 0\} = \{z + \frac{1}{k_1}a_{11}(w,x)=0\}.
\end{equation*}

The coefficient ideal of $(f,k)$ is equivalent to the marked ideal on $N^1$ given by (the restriction to $N^1$ of)
\begin{equation*}
\sum_{i=1}^s \left(\left(X_{i\ell}^{k_i}\right)_{0\leq\ell\leq k_i-1},\,k_i\right)
\end{equation*}
(sum of marked ideals; see \cite[\S3.3]{BMfunct}, \cite[Definition A.8]{BMmin1}), or by
\begin{equation*}
\ucC^1 := \left(\left(X_{i\ell}^{K/k_i}\right)_{\substack{0\leq\ell\leq k_i-1\\i=1,\ldots,s}},\,K\right),
\end{equation*}
where $K$ denotes the least common multiple (or any given common multiple) of $k_1,\ldots,k_q$.

We carry out the construction as in the irreducible case, for each $j=1,\ldots,k-1$. For example,
the monomial part of the coefficient ideal $\ucC^1$ is generated by a monomial
 $w^{K\al_1/k_{i_1}}$, where
\begin{equation*}
\al_1 = k_{i_1} m_{i_1 \ell_1} + \ell_1,
\end{equation*}
as in \eqref{eq:mon1}; etc.
The construction involves successive
nonnegative integers $\widetilde{m}_{i_j \ell_j}$,
each associated to some $X_{i\ell} = X_{i_j \ell_j}$ (except for $i=1, \ell=0$), as above.

We can now apply cleaning blow-ups as above, to successively reduce each $\widetilde{m}_{i_j \ell_j}$
to $0$. We can then introduce new formal (or \'etale) coordinates $y_{10} := X_{10}$ and $y_{i\ell}$, $(i,\ell) \neq (1,0)$,
as above. The effect is to reduce each $f_i(w,x,z)$ to a pre-circulant singularity
\begin{equation*}
\De_{k_i} \left(y_{i0}, w^{n_{i1}+1/k_i} y_{i1}, \ldots, w^{n_{i, k_i-1}+(k_i-1)/k_i}y_{i,k_i-1}\right);
\end{equation*}
i.e., to reduce $f(w,x,z)$ to the product of pre-circulant singularities
\begin{equation*}
\prod_{i=1}^s \De_{k_i} \left(y_{i0}, w^{n_{i1}+1/k_i} y_{i1}, \ldots, w^{n_{i, k_i-1}+(k_i-1)/k_i} y_{i,k_i-1}\right).
\end{equation*}

We can now proceed to reduce each $n_{ij}$ to $0$ by global admissible blowings-up, as in Remark \ref{rem:D_1trick}.
This completes the proof of Theorem \ref{thm:lim}.
\end{proof}

\subsection{Recall on the desingularization invariant}\label{subsec:inv}
Let $X \hookrightarrow Z$ denote an embedded variety ($Z$ smooth), and let $E\subset Z$ denote an snc divisor.

Let $\s: Z' \to Z$ denote a blowing-up with centre $C$. We say that $\s$ is \emph{admissible} (for $(X,E)$)
if $C$ is smooth and snc with respect to $E$, and the Hilbert-Samuel function $H_{X,x}$ is locally constant (as a
function of $x$) on $C$. In the case that $X$ is a hypersurface (i.e., $\dim X = n-1$, where $n=\dim Z$), the latter
property is equivalent to the condition that the order $\ord_xX$ is locally constant on $C$.

Given a sequence of admissible blowings-up
\begin{equation}\label{eq:blup1}
Z = Z_0 \stackrel{\s_1}{\longleftarrow} Z_1 \longleftarrow \cdots
\stackrel{\s_{s}}{\longleftarrow} Z_{t}\,,
\end{equation}
we consider successive transforms $(X_q,E_q)$ of $(X_0,E_0) := (X,E)$: for each $q$, $X_{q+1}$ denotes
the strict transform of $X_q$ by $\s_{q+1}$, and $E_{q+1}$ denotes the divisor whose components are the
strict transforms of all components of $E_q$, together with the \emph{exceptional divisor} $\s_{q+1}^{-1}(C_q)$
of $\s_{q+1}$. We sometimes also call $E_q$ the \emph{exceptional divisor}, in a given \emph{year} $q$.

The desingularization invariant $\inv$ is defined
step-by-step over a sequence of blowings-up \eqref{eq:blup1}, where each successive blowing-up is 
$\inv$-\emph{admissible} 
(meaning that $\inv$ is locally constant on the centre of each blowing-up $\s_{q+1}$).

Assume that $X$ is an \emph{embedded hypersurface}; i.e., $X\hookrightarrow Z$, where $\dim Z = \dim X +1$.

Let $a\in X_q$. Then $\inv(a)$ has the form
\begin{equation}\label{eq:inv}
\inv(a) = (\nu_1(a), s_1(a),\ldots,\nu_t(a), s_t(a), \nu_{p+1}(a)),
\end{equation}
where each $\nu_j(a)$ is a positive rational number (\emph{residual order}) if $j\leq p$,
each $s_j(a)$ is a nonnegative integer (which counts certain components of $E_q$), and
$\nu_{p+1}(a)$ is either $0$ or $\infty$. The successive pairs $(\nu_j(a), s_j(a))$ can be
defined iteratively over \emph{maximal contact subvarieties} of increasing codimension.

For example, in year zero (i.e., if $q=0$), then $\nu_1(a) = \ord_a X$ and
$s_1(a) = \#E(a)$ (the number of components of $E$ at $a$).

The invariant $\inv$ is upper-semicontinous on each $X_q$ (where sequences of the
form \eqref{eq:inv} are ordered lexicographically), and \emph{infinitesimally upper-semicontinuous}
in the sense that $\inv$ can only decrease after blowing up with $\inv$-admissible centre.

For any positive integer $j$, let $\inv_j$ denote the truncation of $\inv$ after the 
$j$th pair $(\nu_j, s_j)$; i.e., $\inv_j(a) = \inv(a)_j$, where the latter means $\inv(a)$
truncated after the pair $(\nu_j(a), s_j(a))$ (and $\inv_j(a) := \inv(a)$ if $j>p$).

The truncated invariant $\inv_j$ can, in fact, be defined step-by-step over a sequence of 
blowings-up \eqref{eq:blup1}, where each successive blowing-up is 
$\inv_j$-\emph{admissible}. Moreover, $\inv_j$ is upper-semicontinuous on each $X_q$
and infinitesimally upper-semicontinuous. Blowings-up that are admissible for $\inv_j$
are not necessarily admissible for $\inv$, but they are admissible for $X$ in the sense 
of Definition \ref{def:admiss}, or for $(X,E)$ in the more general sense of Section \ref{sec:alg}.

\begin{remark}\label{rem:trunc}
If $I$ denotes the maximum value of $\inv$ on a given open set $U$ in $Z_q$, then the
truncation $I_j$ is the maximum value of the truncated invariant $\inv_j$ on $U$
(because of the lexicographic order).
\end{remark}

\begin{remark}\label{rem:memory1}
As stated above, $\inv$ is defined recursively over a sequence of $\inv$-admissible blowings-up; 
more precisely, $\inv$ in year $q$ depends on the previous blowings-up $\s_1,\ldots,\s_q$. This
memory, or dependence on the previous history, is encoded by the $s_j$ entries in $\inv$, which
count the number of components of $E_q$ in certain subblocks of the latter.

In articles on the desingularization algorithm, the notation $\inv = \inv_{X,E}$ is used for $\inv$
as defined recursively over the particular sequence of $\inv$-admissible blowings-up used
in the desingularization algorithm, where the data in any year $q$ depends ultimately only
on the year zero data $(X_0,E_0) = (X,E)$.
\end{remark}

\subsection{Cleaning algorithm}\label{subsec:clean}
An algorithm for resolution of singularities as in \cite{BMinv}, \cite[\S5]{BMfunct} involves factoring
an ideal into its \emph{monomial part} $\cM$, generated locally by a monomial in components of the exceptional
divisor, and \emph{residual part} $\cR$, divisible by no such component; the residual ideal is resolved
first, to reduce to the monomial case where there is a simple combinatorial version of resolution of
singularities.

The cleaning algorithm following reverses this process in a certain sense,
resolving the monomial part directly to obtain a simpler \emph{clean} ideal or singularity. 

\begin{proof}[Proof of Theorem \ref{thm:limintro}]
By Theorem \ref{thm:splitintro}, after blowing up the discrete set of non-$\nc(k)$ points of $S$
finitely many times, we can assume that the ideal of $X$ is generated locally at each non-$\nc(k)$ point
of $S$ by a function \eqref{eq:splitintro} satisfying the conclusion of Theorem \ref{thm:splitintro}.

We will show that the conclusion of Theorem \ref{thm:limintro} holds after a finite sequence
of cleaning blow-ups of $U$, followed by the additional blowings-up of Remark \ref{rem:D_1trick}.

Take $j<k$. Let $T_j$ denote the locus $\{\inv_j = \inv(\nc(k))_j\}$ of points $a\in U$
where $\inv_j(a) = \inv(\nc(k))_j$. Then $T_j$ is a closed subset of $U$, by Remark \ref{rem:trunc}.
Following the proof of the desingularation theorem in \cite{BMfunct}, \cite[Appendix]{BMmin1}, 
$T_j$ is \emph{locally} the cosupport of a marked ideal $\ucC^j = (\cI^j, d^j)$ on a maximal contact
subvariety $N^j$ of codimension $j$ in $U$. (Here $\cI^j$ is an ideal in $\cO_{N^j}$, $d^j$ is a
positive integer, and $\cosupp \ucC^j := \{a\in U: \ord_a\cI^j \geq d^j\}$. See the preceding
references for a detailed exposition of all these notions.)

Let $\cI^j = \cM(\ucC^j)\cdot\cR(\ucC^j)$ denote the factorization of $\cI^j$ into its monomial
and residual parts: the \emph{monomial part} $\cM(\ucC^j)$ is an ideal generated locally by a monomial in components
of the exceptional divisor transverse to $N^j$, and the \emph{residual ideal} $\cR(\ucC^j)$ is divisible by no such component.
Let $\ucM(\ucC^j)$ denote the marked ideal $(\cM(\ucC^j), d^j)$. The $\cosupp \ucM(\ucC^j) \subset \cosupp \ucC^j$,
and any sequence of blowings-up that is admissible for the marked ideal $\ucM(\ucC^j)$ is also admissible for $\ucC^j$
(where a blowing-up is \emph{admissible} for a marked ideal if its centre lies in the cosupport and is snc with the
exceptional divisor).

The exponents (each divided by $d_j$) of a local monomial generator of $\ucM(\ucC^j)$ are invariants
of $(X,E)$ (in particular, independent of the choice of a local maximal contact subvariety), 
by \cite[Thm\,6.2]{BMfunct}. By combinatorial
resolution of singularities in the monomial case, there is an invariantly defined $\inv_j$-admissible sequence
of blowings-up, after which $\cosupp \ucM(\ucC^j) = \emptyset$.

We call the blowings-up in such a sequence \emph{cleaning blow-ups}. The centres of these
cleaning blow-ups are invariantly defined closed subsets of $\{\inv_j = \inv(\nc(k))_j\}$.

To complete the proof of Theorem \ref{thm:limintro}, we apply the preceding cleaning algorithm
successively, for each $j=k-1,\ldots,1$, and afterwards make the additional blowings-up described in Remark \ref{rem:D_1trick}. 
The cleaning blow-ups involved coincide with those described
locally in the proof of Theorem \ref{thm:lim}. So Theorem \ref{thm:limintro} follows from Theorem \ref{thm:lim}.
\end{proof}

\begin{definition}\label{def:clean}
We say that the coefficient ideal $\ucC^j$ above is \emph{clean} at a point $a$ if $a\notin \cosupp \ucM(\ucC^j)$.
We say also that $X$ is \emph{clean} at $a$ if $\ucC^j$ is clean at a $a$, for $j=1,\ldots,k-1$.
\end{definition}

\begin{remark}\label{rem:clean}
Normal crossings singularities, for example, are clean, and circulant singularities
are clean, according to the proof of Theorem \ref{thm:lim}.

Theorem \ref{thm:limintro} will be applied in Section \ref{sec:alg} to an open set $U$, where
$U$ is the complement of a closed algebraic (or analytic) set $\Sigma$, and $X$ is clean
in a neighbourhood of $\Sigma$ in $U$. In this situation, the centres of the blowings-up involved in
Theorem \ref{thm:limintro} are smooth closed subsets of $X$ containing no clean points.
\end{remark}

\begin{remark}\label{rem:cleanlem}
The transform $\ucM(\ucC^j)'$ of $\ucM(\ucC^j)$ by a cleaning blow-up does not, in general,
coincide with the monomial part $\ucM((\ucC^j)')$ of the transform of $\ucC^j$ because the exceptional
divisor may factor from the pull-back of $\cR(\ucC^j)$. So monomial desingularization of $\ucM(\ucC^j)$
does not guarantee that $\ucC^j$ becomes clean.

The \emph{cleaning lemma} \cite[Lemma 2.1]{BMmin1} provides simple sufficient conditions
(which are satisfied in Theorem \ref{thm:limintro}) for desingularization of $\ucM(\ucC^j)$ to lead to
a clean ideal $\ucC^j$. We do not need the cleaning lemma in the proof of Theorem \ref{thm:limintro}
because the explicit local computation in the proof of Theorem \ref{thm:lim} shows that all $\ucC^j$ become
clean. So we do not go into the details of the preceding paragraph.
\end{remark}

\smallskip
\section{Limits of triple normal crossings}\label{sec:triplenc}

A proof of our main
Conjecture \ref{conj:min} following the approach of this article \emph{requires} an analogue of
Theorem \ref{thm:limintro} for any $k < n$. We provide this here
for $k\leq 3$; see Theorem \ref{thm:lim3intro} following. The analogue of Theorem \ref{thm:lim3intro}
in the case $k=2$ is much simpler, and is also proved in \cite{BMmin1}. In particular, in the case that
$n=5$, these results together with Theorem \ref{thm:limintro} give the required analogue of the latter
for any $k<n$. We will use this in Section \ref{sec:alg} to prove Theorems \ref{thm:dim4} and \ref{thm:nc(3)}.

\begin{theorem}\label{thm:lim3intro}
Consider an embedded hypersurface $X\hookrightarrow Z$.
Let $U$ denote an open subset of $Z$. Assume that (after an $\inv$-admissible sequence of blowings-up)
the maximum value of $\inv$
on $U$ is $\inv(\nc(3)) = (3,0,1,0,1,0,\infty)$, so that the stratum $S := \{\inv = \inv(\nc(3))\}$ is a smooth
subvariety of dimension $n-3$ in $U$, where $n = \dim Z$. 
Suppose $X$ is generically $\nc(3)$ on $S$. Then there is a finite sequence of $\inv_1$-admissible blowings-up of $U$,
preserving the $\nc(3)$-locus, after which $X$
is a product of circulant singularities at every point of (the strict transform of) $S$.
\end{theorem}

\begin{remark}\label{rem:lim3intro}
Under the hypotheses of Theorem \ref{thm:lim3intro}, the non-$\nc(3)$ points of $X$ in $S$
form a proper closed algebraic (or analytic) subset of $S$ (by Lemma \ref{lemma:genericnc}).
It follows from resolution of singularities of this subset that, after a finite number of $\inv$-admissible
blowings-up, we can assume that every non-$\nc(3)$ point $a$ of $S$ has an \'etale (or analytic) 
neighbourhood in $Z$ with coordinates $(w,u,x,z) = (w_1,\ldots,w_r,u_1,\ldots,u_q,x_1,x_2,z)$ in
which $\{w_i=0\}$, $i=1,\ldots,r$, are the components of $E$ at $a=0$, $S = \{z=x=0\}$,
$X$ is $\nc(3)$ on $S\backslash \{w_1\cdots w_r=0\}$, and the ideal of $X$ is generated
by a function
\begin{equation}\label{eq:triplenc}
f(w,u,x,z) = z^3 -3B(w,u,x)z + C(w,u,x),
\end{equation}
where the coefficients $B,\,C$ are regular (or analytic) functions,
$f$ is in the ideal generated by $x_1,x_2,z$, and $f$ splits formally
(into three factors of order $1$) at every point where $z=x=0$ and $w_1\cdots w_r \neq 0$.
\end{remark}

Theorem \ref{thm:lim3intro} then follows from Proposition \ref{prop:lim3} below, which is 
an analogue of Theorem \ref{thm:lim}. Proposition \ref{prop:lim3} will be stated using local hypotheses as in
Theorem \ref{thm:lim}, with globalization via the cleaning algorithm, as in
\S\ref{subsec:clean}, and the analogue of Remark \ref{rem:D_1trick}.

\subsection{Splitting}\label{subsec:splitting}
\begin{proposition}\label{prop:limnc3}
Let $f$ denote a function as in \eqref{eq:triplenc} (satisfying the conditions given in Remark \ref{rem:lim3intro}).
Assume, moreover, that
\begin{equation}\label{eq:invnc3}
\inv(0) = \inv(\nc(3)) = (3,0,1,0,1,0,\infty),
\end{equation}
and that $\{w_1\cdots w_r = 0\}$ is the exceptional divisor.
Then, after a finite number of blowings-up with ($\inv$-admissible) centres of the form
$\{z=x=w_j=0\}$, for some $j$, we can assume that $f(v_1^6,\ldots,v_r^6, u,x,z)$ splits.
\end{proposition}

\begin{remark}\label{rem:onDsq}
More precisely, there is a finite sequence of admissible blowings-up of $U$, as indicated, after which the
conclusion of Proposition \ref{prop:limnc3} holds at every point of $S\cap E$, where
$f$ is a local generator of the ideal of $X\subset Z$: this is a consequence of the fact that the centres of blowing up
involved in Lemma \ref{lemma:Dsq} (as used in the proof following) are each given
by the intersection of $S$ with a component of the exceptional divisor.
\end{remark}

\begin{remark}\label{rem:limnc3}
The proof of Proposition \ref{prop:limnc3} uses only $\inv(0) = (3,0,1,\ldots)$, instead of all the
information given by \eqref{eq:invnc3}.
\end{remark}

\begin{remark}\label{rem:limnc3E}
In Section \ref{sec:alg}, we will use a version of Proposition \ref{prop:limnc3} for a product of $f$ as above
with generators $y_1,\ldots,y_r$ of the ideals of the components of $E$ at $a$. See Remarks \ref{rem:Dsq}
and \ref{rem:limE}. In this context, splitting as in Proposition \ref{prop:limnc3} holds after a finite number of $\inv$-admissible
blowings-up with centres of the form $\{z=x=y=w_j=0\}$.

It is not difficult to extend Proposition \ref{prop:limnc3} to a product of functions $f$, each of order $\leq 3$;
in this article, we will need only the preceding situation, or the simpler version for $f$ of order $2$ and the components of $E$.
\end{remark}

\begin{proof}[Proof of Proposition \ref{prop:limnc3}]
The discriminant $D$ of $f$ is given by
$$
D = - \frac{1}{27}\left(C^2 - 4B^3\right).
$$
By Lemma \ref{lemma:Dsq},
after a finite number of blowings-up with centres of the form $\{z=x=w_j=0\}$, for some $j$, we can assume
that $\De(v_1^2,\ldots, v_r^2, u, x)$ is a square. The result follows from Lemma \ref{lem:3.3} below.
\end{proof}

The following two lemmas are essentially Lemmas 3.3 and 3.6 in \cite{BLMmin2}. 
We will use the fact that the first
coefficient (marked) ideal of the marked ideal $(f,3)$ is 
$$
I := \left((B^3, C^2), 6\right) =  \left((C^2, D), 6\right).
$$
Since $\inv(a) = (3,0,1,\ldots)$, we have $I = w^\ga \tI$, where $w^\ga$ is a monomial and
$\tI$ has order 6 at $0$.

Set
\begin{equation*}
R := \IC\llb w,u, x\rrb,\quad
S := \overline{\IC\llbr w\rrbr}\llb u, x\rrb.
\end{equation*}
Then $f$ splits in $S[z]$; say,
$$
f = (z + b_0)(z + b_1)(z + b_2).
$$
Moreover, each $b_j$ belongs to the ideal $(x)$. 
Define
\begin{equation*}
\eta_i := \frac{1}{3}\sum_{j=0}^2 \ep^{ij}(z+ b_j), \quad i = 0,1,2,
\end{equation*}
where $\ep =  e^{2\pi i/3}$. Then $\eta_0 = z$ and
\begin{align}\label{eq:factors}
f &= \prod_{i=0}^2\left(z + \ep^i\eta_1 + \ep^{2i}\eta_2\right)\\
  &= z^3 -3\eta_1\eta_2 z + \eta_1^3 + \eta_2^3 \notag
\end{align}
in $S[z]$. In particular,
$$
B = \eta_1\eta_2, \quad C = \eta_1^3 + \eta_2^3, \quad D = -\frac{1}{27}\left(\eta_1^3 - \eta_2^3\right)^2
$$
in $S$.

\begin{lemma}\label{lem:3.3}
Assume that $D$ is a square in $R$. Then $f(v_1^3,\ldots,v_r^3, u,x,z)$ splits.
\end{lemma}

\begin{proof}
Write $D = A^2 \in R$; we can take $A = \eta_1^3-\eta_2^3$. Recall that
$I = (B^3, C^2) = (D, C^2)= w^\ga \tI$, as above. 
Then  $w^\ga$ is the monomial in $w$ of largest exponent which factors from both $A^2,\,C^2$.
Therefore, each $\ga_k$ is even; say $\ga = 2\al$.

We have $4B^3 = (C-A)(C+A)$.

We claim that $w^{-\al}C$ and $w^{-\al}A$ are relatively prime in R. Indeed,
it is easy to check they are relatively prime in $S$ since $A = \eta_1^3 - \eta_2^3$, $C = \eta_1^3 + \eta_2^3$, and the ideal $(\eta_1,\eta_2)=(x_1,x_2)$ in $S$.
Since $\tilde I$ has order $6$, either $\ord\, w^{-\ga} D = \ord_{x} w^{-\ga} D$ or 
$\ord\, w^{-\ga} C^2 = \ord_{x} w^{-\ga} C^2$. In either case, we can use Lemma
\ref{lem:3.6} following to conclude that $w^{-\al}C$, $w^{-\al}A$ are 
relatively prime in R.

Therefore, $w^{-\de}(C-A) = 2w^{-\de}\eta_2^3$ and $w^{-\de}(C+A) = 2w^{-\de}\eta_1^3$
are relatively prime in $R$, where $\de$ denotes the largest exponent of a monomial in $w$ that divides
$C-A$ and $C+A$.  Moreover, the product $C^2 - A^2 = 4 w^{-2\de} B^3$ is a cube times a monomial
$w$ in $R$. Hence both $\eta_1^3$ and $\eta_2^3$ are cubes (times monomials in $w$)
in $R$. By \eqref{eq:factors}, $f(v_1^3,\ldots, v_r^3,u,x,z)$ splits in 
$\IC\llb v,u,x\rrb [z]$ and the result follows.
\end{proof}

\begin{lemma}\label{lem:3.6}
Let $G \in R$. Suppose that $\ord\, G = \ord_x G$. Let $\theta \in R$ be a nonunit
which divides $G$. Then $\theta$ is also a nonunit in $S$.
\end{lemma}

\begin{proof}
Consider a decomposition of $G$ into irreducible factors in $R$, $G = \prod \theta_i^{m_i}$,
where the $m_i$ are positive integers. For all $i$, $\ord\, \theta_i = \ord_x \theta_i$. By the
hypothesis, $\sum m_i\ord\, \theta_i  = \sum m_i\ord_x \theta_i$. Therefore,
$\ord\,\theta_i  = \ord_x \theta_i$, for all $i$. The result follows.
\end{proof}

\subsection{Splitting exponents}\label{subsec:splitexp}
Consider $f(w,u,x,z)$ as in \eqref{eq:triplenc}.
Assume that $f$ is irreducible and that $f(v_1^{p},\ldots,v_r^{p},u,x,z)$ splits, for some $p$. Let $S_3$ denote the 
group of permutations of the the roots of $f(v_1^{p},\ldots,v_r^{p},u,x,z)$. Then
$(\IZ_{p})^r$ maps onto a subgroup of $S_3$ which acts transitively on the
roots (since $f$ is irreducible). It follows (as in the proof of Lemma \ref{cor:finitesplit}(2)) that $f(v_1^{q_1},\ldots,v_r^{q_r},u,x,z)$ splits, 
where, for each $i=1,\ldots,r$, $q_i \leq 3$ and the group $\IZ^{(i)}_{q_i} := \{1\}^{i-1} \times\IZ_{q_i}\times\{1\}^{r-i}$ maps
 onto a cyclic subgroup  $\IZ_{q_i}$ of $S_3$; 
moreover, $\IZ_{q_1} \times \cdots \times \IZ_{q_r}$ maps onto a subgroup $G$ of $S_3$ that 
acts transitively on the roots of $f(v_1^{q_1},\ldots,v_r^{q_r},u,x,z)$.
 
 \begin{lemma}\label{lem:irred2pt}
Suppose $f(w,u,x,z)$ is irreducible and $f(v_1^{q_1},\ldots,v_r^{q_r},u,x,z)$ splits, where $q_1,\ldots,q_r$ 
are chosen as above. Then $q_i = 1$ or $q_i = 3$, for each $i=1,\ldots,r$. 
\end{lemma}

\begin{proof}
For each $i$, the translates of the subgroup $\IZ_{q_i}$ of $S_3$ above by the elements of $G$ provide a partition
of the set of roots of $f(v_1^{q_1},\ldots,v_r^{q_r},u,x,z)$ into subsets of $q_i$ elements; therefore, $q_i$ divides $3$.
\end{proof}

\begin{example}\label{ex:irred2pt}
Consider the splitting 
$$
z^2 -w_1w_2x^2 = \left(z-w_1^{1/2} w_2^{1/2}x\right)\cdot\left(z+w_1^{1/2} w_2^{1/2}x\right).
$$
Both $q_1=2$ and $q_2=2$ are needed for a splitting, although each $\IZ_{q_i}^{(i)}$ maps onto
the cyclic group $\IZ_2  = S_2$ (which acts transitively on the roots).
\end{example}

\subsection{Circulant normal form}\label{subsec:circ3}

Theorem \ref{thm:lim3intro} is a consequence of the following result, which we prove in this subsection.

\begin{proposition}\label{prop:lim3}
Assume that (after an $\inv$-admissible sequence of blowings-up) $X\subset Z$ is defined locally at a given
point by a function 
$$
f(w,u,x,z) = f(w_1,\ldots,w_r,u_1,\ldots,u_q,x_1,x_2,z)
$$
as in \eqref{eq:triplenc}, where
$f$ is $\nc(3)$ on $\{z=x=0,\,w_1\cdots w_r \neq 0\}$, $\{w_1\cdots w_r = 0\}$ is the exceptional divisor,
and $\inv(0) = \inv(\nc(3))$. Then there is a finite sequence of $\inv_1$-admissible blowings-up that are isomorphisms
over the $\nc$ locus, after which the only non-$\nc(3)$ singularities in the stratum $S$ given by the closure
of the $\nc(3)$ points are products of circulant singularities.

In particular, if $f$ is irreducible, then we reduce to the case that the only non-$\nc(3)$ singularities in $S$
are circulant singularities $\De_3(z, w_i^{1/3}y_1, w_i^{2/3}y_2)$, for some $i=1,\ldots,r$.
\end{proposition}

\begin{remark}\label{rem:lim3E}
There is again a more general statement involving products of circulant singularities as in Proposition
\ref{prop:lim3} with generators of the the ideals of the components of a simple normal crossings divisor $E$.
See Remarks \ref{rem:limE} and \ref{rem:limnc3E}.
\end{remark}

\begin{proof}
We will prove the result for $f$ irreducible, and make a remark at the end about the general case.

\medskip\noindent
\emph{Irreducible case.} We follow the outline of the proof of Theorem \ref{thm:lim}. By Lemma \ref{lem:irred2pt} (and Remark \ref{rem:onDsq}), 
we can assume that the function 
$f(v_1^3,\ldots, v_s^3,w_{s+1},\ldots, w_r, u,x,z)$ splits, for some $s$, $1\leq s\leq r$, and has
zeros (as a polynomial in $z$) of the form $-b(\ep^{\ell_1}v_1,\ldots,\allowbreak \ep^{\ell_s}v_s, t,u,x)$, where $\ep=e^{2\pi i/3}$
and $t:= (w_{s+1},\ldots, w_r)$. Set
\begin{equation*}
Y_\ell := z + b(\ep^\ell v_1,v_2, \ldots, v_s, t,u,x), \quad \ell = 0,1,2,
\end{equation*}
and define $X_0,X_1,X_2$ by \eqref{eq:inveigenmatrix}; i.e.,
\begin{align*}
X_0 &= \frac{1}{3}\sum_{j=0}^{2} Y_j = z,\nonumber \\
X_\ell &= \frac{1}{3} \sum_{j=0}^{2} \ep^{\ell(3-j)} Y_j =  \frac{1}{3} \sum_{j=0}^{2} \ep^{\ell(3-j)} b(\ep^j v_1,v_2,\ldots,v_s,t,u,x),\quad \ell =1,2.
\end{align*}

For each $\ell =1,2$, $v_1^{3-\ell} X_\ell$ is invariant under the action of the group $\IZ_3$ of cube
roots of unity, induced by $(v,t,u,x)\mapsto (\ep v_1,v_2,\ldots,v_s,t,u,x)$, so that
\begin{equation*}
v_1^{3-\ell} X_\ell = \eta_\ell (v_1^3,v_2,\ldots,v_s,t, u,x),
\end{equation*}
where $\eta_\ell(w_1,v_2,\ldots,v_s,t,u,x) \in \IC\llb w_1,v_2,\ldots,v_s,t,u,x\rrb$, $\ell =1,2$. Since each $\eta_\ell$ must, therefore, be divisible by $v_1^3$,
we can write
\begin{equation*}
X_\ell = v_1^{3 m_{1\ell} + \ell} \zeta'_\ell (v_1^3,v_2, \ldots,v_s,t,u,x) = w_1^{m_{1\ell} + \ell/3} \zeta^{(1)}_\ell(w_1,v_2,\ldots,v_s,t,u,x),
\end{equation*}
where $m_{1\ell}$ is a nonnegative integer and $\zeta^{(1)}_\ell(w_1,v_2,\ldots,v_s,t,u,x) \in \IC\llb w_1,v_2,\ldots,v_s,\allowbreak t,u,x\rrb$ is not divisible by $w_1$, $\ell =1,2$.

Likewise, the roots of $f(v_1^3,v_2^3,v_3^3,\ldots,v_s^3,t,u,x,z)=0$ are permuted by the action of $\IZ_3$ 
induced by $(v,t,u,x)\mapsto (v_1,\ep v_2,v_3,\ldots,v_s,t,u,x)$, and it follows that $X_\ell$ can be written
\begin{equation*}
X_\ell = w_1^{m_{1\ell} + \ell/3} w_2^{m_{2\ell} + q_{2\ell}/3}\zeta^{(2)}_\ell(w_1,w_2,v_3,\ldots,v_s,t,u,x),\quad \ell = 1,2,
\end{equation*}
where $\zeta^{(2)}_\ell$ is divisible by neither $w_1$ nor $w_2$, and $\{q_{21},q_{22}\} = \{1,2\}$.

We repeat this process for $w_3,\ldots,w_s$, and conclude that
\begin{align*}
X_\ell &= w_1^{m_{1\ell} + \ell/3} w_2^{m_{2\ell} + q_{2\ell}/3} \cdots w_s^{m_{s\ell} + q_{s\ell}/3} \zeta^{(s)}_\ell(w,u,x)\\
          &= w_1^{m_{1\ell} + \ell/3} w_2^{m_{2\ell} + q_{2\ell}/3} \cdots w_s^{m_{s\ell} + q_{s\ell}/3} t^{n_\ell} \zeta_\ell(w,u,x),\quad \ell = 1,2,
\end{align*}
where $t^{n_\ell}$ is a monomial in $t = (w_{s+1},\ldots,w_r)$ (with integral exponents),
$\zeta_\ell$ is divisible by no $w_i$, $i=1,\ldots,r$, and each $\{q_{i1},q_{i2}\} = \{1,2\}$.

As in the proof of Theorem \ref{thm:lim}, the coefficient ideal of the marked ideal $(f,3)$ is equivalent to the marked ideal
\begin{multline*}
\ucC^1 := \left( \left(w_1^{3m_{11} + 1} w_2^{3m_{21} + q_{21}} \cdots w_s^{3m_{s1} + q_{s1}} t^{3n_1} \zeta_1^3,\right.\right.\\
\left.\left. w_1^{3m_{12} + 2} w_2^{3m_{22} + q_{22}} \cdots w_s^{3m_{s2} + q_{s2}}  t^{3n_2} \zeta_2^3\right), 3\right)
\end{multline*}
on the maximal contact subspace $N^1 := \{z=0\}$. 
Since $\inv(0) = (3,0,1,\ldots)$, it follows that the exponent $r$-tuple of one of the two monomials in $w$
in $\ucC^1$ (which we denote $w^\ga$) is less than the other (denoted $w^{\ga + \de}$), and the $\zeta_\ell$ 
corresponding to the first (say, $\zeta_{\ell_1}$, where $\ell_1 = 1$ or $2$)
has order $1$.

We can then apply a cleaning procedure in the proof
of Theorem \ref{thm:lim}.
We first blow up with combinatorial centres $\{w_{i_1} = \cdots = w_{i_p} =0\}$, where $p\leq 3$, in the maximal
contact subspace $N^2 = \{z=\zeta_{\ell_1}=0\}$ to reduce to $\de = (\de_1, \ldots, \de_s, 0,\ldots,0)$ with 
$|\de| = \de_1+\cdots +\de_s < 3$. Note that $|\de| < 3$ implies that $w^\de$ depends on at most two variables $w_i$. 
Moreover, using the fact that each $\{q_{i1},q_{i2}\} = \{1,2\}$ above, it is easy to see this implies also that we
have modified our expression for $\ucC^1$ in such a way that now $s\leq 2$.

For the second cleaning step, we can now
blow up with codimension one centres $\{w_i=0\}$ in $N^1 = \{z=0\}$ (which preserve $|\de|<3$) to get also 
$\ga = (\ga_1,\ldots,\ga_s,0,\ldots,0)$, where $s\leq 2$ and
$|\ga| < 3$. Set $\al := \ga/3$, $\be := \de/3$.

We conclude that, after cleaning,
$f$ can be written as
\begin{equation*}
\De_3\left(z, w^\al y_1, w^{\al + \be} y_2\right), 
\end{equation*}
where $y_1, y_2$ are suitable \'etale (or analytic) coordinates (as in the proof of Theorem \ref{thm:lim}), $w^\al$ and $w^\be$ are each
monomials in $w_1^{1/3},\ldots, w_s^{1/3}$ of order $< 1$, and
$$
\De_3\left(z, w^\al y_1, w^{\al + \be} y_2\right) = z^3 + w^{3\al} y_1^3 + w^{3(\al+\be)} y_2^3 - 3w^{2\al+\be}y_1y_2z
$$
is a polynomial in $(w,y,z)$; i.e., $w^{2\al+\be}$ has integral exponents (cf. \eqref{eq:cp3expansion}). 
The only possibilities for $w^\al$, $w^\be$ satisfying these conditions are
$$
w^\al\  =\ 1,\ \ w_1^{1/3},\ \ w_1^{2/3}\ \ \text{or}\ \ w_1^{1/3}w_2^{1/3}
$$
(after reordering the $w_i$ if necessary), and then
$$
w^\be\  =\ 1,\ \ w_1^{1/3},\ \ w_1^{2/3}\ \ \text{or}\ \ w_1^{1/3}w_2^{1/3}\ \ \text{(respectively)}.
$$

In other words, after cleaning, we reduce to four possible cases:
\begin{align*}
&\De_3(z,y_1,y_2):\qquad \nc(3),\\
&\De_3(z,w_1^{1/3}y_1,w_1^{2/3}y_2):\qquad \cp(3),\\
&\De_3(z,w_1^{2/3}y_1,w_1^{4/3}y_2) \qquad\ \  \text{(cf. Example \ref{ex:minprecirc})},\\
&\De_3(z,w_1^{1/3}w_2^{1/3}y_1,w_1^{2/3}w_2^{2/3}y_2). 
\end{align*}

In particular, we have either $\cp(3)$ or one of the pre-circulant third and fourth cases at every point of $S\cap E$, where $S$ denotes
the closure of the $\nc(3)$-locus.
The third and fourth cases can be handled as in Remark \ref{rem:D_1trick}.
For both of these cases, we first blow up with centre given by the non-$\cp(3)$ points of $S\cap E$.

In the third case, this means that (locally) we first
blow up with centre $\{z=y_1=y_2=w_1=0\}$ to introduce the divisor $D_1$. 
Afterwards, we blow up with centre given by $\{z=y_1=w_1=0\}$ in the $w_1$-chart---this extends
to a global smooth centre in $D_1$ given by a component of the intersection of $D_1$ with the locus of points of order $3$ of $f$ or $X$---and
we thereby reduce to $\cp(3)$.

In the fourth case, let $E_i$ denote the component $\{w_i=0\}$ of $E$, $i=1,2$.
The first blowing-up above means that (locally) we introduce $D_1$ by blowing up $\{z=y_1=y_2=w_1=w_2=0\}$. Then $\De_3(z,w_1^{1/3}w_2^{1/3}y_1,w_1^{2/3}w_2^{2/3}y_2)$ transforms to 
\begin{equation}\label{eq:case4}
\De_3(z,w_1^{2/3}w_2^{1/3}y_1,w_1^{4/3}w_2^{2/3}y_2)
\end{equation} 
in the $w_1$-chart, and a symmetric expression in the $w_2$-chart. 
We now blow up with centre given by $\{z=y_1=w_1=0\}$ in the $w_1$-chart and by $\{z=y_1=w_2=0\}$ in the $w_2$-chart; again this extends globally to a smooth centre given by a
component of the intersection of $D_1$ with the order $3$ locus of $X$. (More precisely, the latter intersection is $\{z = w_1 = y_1w_2 =0\}$ in the
$w_1$-chart, for example, and we are blowing up the irreducible component not contained in $D_1 \cap E_2$.)
In the new $w_1$-chart of the latter
blowing-up of \eqref{eq:case4}, we get $\De_3(z,w_1^{2/3}w_2^{1/3}y_1,w_1^{1/3}w_2^{2/3}y_2)$.
After a further blowing-up with centre $\{z=w_1=w_2=0\}$ (globally, $X\cap D_1 \cap E_2$), we have only $\cp(3)$ points.

\medskip\noindent
\emph{General case.} In the case that $f$ is not irreducible, we can also follow the proof of Theorem \ref{thm:lim}.
The result of cleaning in this case is to reduce $f$ already to a product of circulant singularities; i.e., to either
$\nc(3)$ or 
\[
\smooth \times \cp(2):\qquad y_{10}\,\De_2(y_{20}, w^{1/2}y_{21}) = y_{10}\,(y_{20}^2 - wy_{21}^2).
\]
\end{proof}

\smallskip
\section{Partial desingularization algorithm}\label{sec:alg}
Let $X \hookrightarrow Z$ denote an embedded variety ($Z$ smooth), and let $E\subset Z$ denote an snc divisor.
We say that $(X,E)$ is \emph{normal crossings} (nc) at a point $a$ if $X\cup E$ is nc at $a$.

Let $\s: Z' \to Z$ denote a blowing-up with centre $C$. We say that $\s$ is \emph{admissible} (for $(X,E)$)
if $C$ is smooth and snc with respect to $E$, and the Hilbert-Samuel function $H_{X,x}$ is locally constant (as a
function of $x$) on $C$ (cf. Definitions \ref{def:nc}, \ref{def:admiss}).
In the case that $X$ is a hypersurface (i.e., $\dim X = n-1$, where $n=\dim Z$), the latter
property is equivalent to the condition that the order $\ord_xX$ is locally constant on $C$.

Given a sequence of admissible blowings-up
\begin{equation}\label{eq:blup}
Z = Z_0 \stackrel{\s_1}{\longleftarrow} Z_1 \longleftarrow \cdots
\stackrel{\s_{t}}{\longleftarrow} Z_{t}\,,
\end{equation}
we consider successive transforms $(X_j,E_j)$ of $(X_0,E_0) := (X,E)$,
as in \S\ref{subsec:inv}.
In a given \emph{year} $j$, it will often be convenient to drop the index $j$ and simply write $M,X,E$ instead
of $M_j,X_j,E_j$.

\begin{theorem}\label{thm:mainA}
Assume that $\dim X \leq 4$. Then there is a finite sequence of admissible blowings-up \eqref{eq:blup}
such that every $\s_j$ is an isomorphism over the nc locus of $(X_0,E_0) = (X,E)$, and $X_t$
has smooth normalization.
\end{theorem}

Theorem \ref{thm:mainA} is a corollary of the following more precise result.

\begin{theorem}\label{thm:mainB}
Assume that $\dim X \leq 4$. Then there is a finite sequence of admissible blowings-up \eqref{eq:blup}
such that every $\s_j$ is an isomorphism over the nc locus of $(X_0,E_0) = (X,E)$, and $(X_t, E_t)$
has only \emph{minimal singularities}.
\end{theorem}

There is an analogous version of Theorem \ref{thm:nc(3)} for a pair $(X,E)$, where we preserve
normal crossings singularities of $(X,E)$, i.e., of $X\cup E$, of order at most three, in any dimension. 
See Theorem \ref{thm:main3} below.

Normal crossings singularities and, more generally, minimal singularities, are hypersurface singularities.
The class of \emph{minimal singularities} denotes the class of products of circulant singularities
(as given by Theorems \ref{thm:limintro}, \ref{thm:lim}) and their neighbours. A \emph{neighbour} of a circulant singularity
means either a singularity that occurs in a small neighbourhood of the latter, or a limit of singularities
in a neighbourhood which cannot be eliminated. (See \S\ref{subsec:cp4summary}.) There are finitely many
minimal singularities (up to \'etale isomorphism) in Theorem \ref{thm:mainB}.

The class of \emph{minimal singularities} of $(X,E)$ means the class of minimal singularities of $X\cup E$.

Minimal singularities for $\dim X \leq 4$ have smooth normalization (see \S\ref{subsec:cp4summary}, Remark \ref{rem:cp3,2}), 
so that Theorem \ref{thm:mainA} is an immediate
consequence of Theorem \ref{thm:mainB}.

This section is devoted mainly to a proof of Theorem \ref{thm:mainB}, though the first steps below apply to any dimension.
In the case of Theorem \ref{thm:main3}, the entire argument follows 
parts of the proof of Theorem \ref{thm:mainB} that apply to any dimension, and we will add detail 
in \S\ref{subsec:MinimalHigherDim}.

In general dimension, we can reduce the theorems
to the case that $X$ is an \emph{embedded hypersurface} (i.e., $X\hookrightarrow Z$, where
$n:= \dim Z = \dim X +1$) using the standard desingularization algorithm. Indeed, the Hilbert-Samuel function
$H_{X,x}$ determines the local \emph{minimal embedding dimension} $e_{X,x} = H_{X,x}(1) - 1$, so that
the desingularization algorithm first eliminates points of embedding codimension $> 1$ without modifying
nc points.

So from now on, we assume that $X$ is an embedded hypersurface. 

\subsection{Invariant for a normal crossings singularity}\label{subsec:invnc}
Let $a\in X$. Set $p := \ord_a X$ and
$r := \#E(a)$ (the number of components of $E$ at $a$). 
We will call $(p,r)$ the \emph{order} of $(X,E)$ at $a$.
The order $(p,r)$, as a function of the point $a$, is upper-semicontinuous with respect to the 
lexicographic ordering of pairs $(p,r)$.

If $(X,E)$ has order $(p,r)$ and is nc at $a$, then the desingularization
invariant
\begin{equation}\label{eq:ncinvt}
\inv(a) = \inv_{X,E}(a) = (p,r,1,0,\ldots,1,0,\infty),
\end{equation}
where there are $p+r$ pairs (before $\infty$).
Note that $\inv_{X,E}$ here is the desingularization invariant in \emph{year zero} (before we begin blowing up; see \S\ref{subsec:inv}
and also \cite[\S{A.2}]{BMmin1}).
The condition \eqref{eq:ncinvt} does not, in general, imply that $(X,E)$ is nc at $a$, as explained
by an example in \S\ref{subsec:split}.
See \cite[Thm.\,3.4]{BDMV} for a more precise
statement about the invariant at an nc point. (Note that nc is snc in an \'etale neighbourhood.)

Let $\inv_{p,r}$ denote the right-hand side of \eqref{eq:ncinvt}, so that, in particular, $\inv_{p,0} = \inv(\nc(p))$.
Given a sequence of $\inv$-admissible blowings-up \eqref{eq:blup} and a pair of nonnegative integers
$(p,r)$, let $S_{p,r}$ denote the $\inv_{p,r}$-\emph{stratum} in a given year $j$; i.e., the locus of
points where $\inv = \inv_{p,r}$ in year $j$.

Note that, at a point $a\in S_{p,r}$ in a year $j>0$, $\ord_a X_j = p$, but the order of $(X_j,E_j)$ may be greater
than $(p,r)$ because the order of $(X_j,E_j)$ counts all components of $E_j$ at $a$, while $r$ in $\inv_{p,r}$
counts only \emph{old} components of $E_j$ (see \cite[Remark A.18]{BMmin1}).

\begin{remarks}\label{rem:ncstratum} 
(1) If $p+r > n = \dim Z$, then $S_{p,r} = \emptyset$. If $p+r = n$, then $S_{p,r}$ is a discrete subset
of $X\cup E$. 

\medskip\noindent
(2) In a year $j>0$, $(X_j,E_j)$ (or $(X,E)$, in our shorthand language above)
need not be nc at a point of a stratum $S_{p,r}$ even if $(X_j,E_j)$ is
generically nc on $S_{p,r}$; e.g., circulant singularities may occur. On the other hand,
$(X_j,E_j)$ is nc at every point of $S_{1,r}$.

\medskip\noindent
(3) Theorems \ref{thm:mainA} and \ref{thm:mainB} preserve normal crossings points of $(X,E) = (X_0,E_0)$,
but not necessarily normal crossings points of $(X_j,E_j)$, $j\geq 1$. 
\end{remarks}

\subsection{Overview of the proof}\label{subsec:overviewproof}
Given $n$, let $I_n$ denote the finite lexicographic sequence of pairs $(p,r)$, where $p+r\leq n$. Our proof of
Theorem \ref{thm:mainB} (in the case $n=5$, say) will be presented as a recursive or iterative algorithm involving
successive modification of non-$\nc$ points of the strata $S_{p,r}$, $(p,r) \in I_n$, in decreasing order; i.e., beginning
with $(p,r) = (5,0)$ and terminating with the base case $(p,r) = (1,0)$.

To prove Conjecture \ref{conj:min} in general, we would need the corresponding argument by
induction over the sequence $I_n$, for any $n$, beginning with the base case $(1,0)$. The inductive claim can be
formulated as follows.

\begin{claim}\label{claim:ind}
Given $(p,r) \in I_n$, there is an admissible sequence of blowings-up \eqref{eq:blup}, satisfying the following conditions:
\begin{enumerate}
\item each blowing-up is an isomorphism over the locus of normal crossings points of $(X_0,E_0)$ of order at most $(p,r)$;

\smallskip
\item over any open subset $U$ where $(X,E)|_U$ is normal crossings, \eqref{eq:blup} coincides with the blow-up sequence given by the desingularization algorithm, stopped when $\inv \leq \inv_{p,r}$;

\smallskip
\item $(X_t,E_t)$ has only minimal singularities.
\end{enumerate}
\end{claim}

We will make a concluding remark on a strategy to prove Claim \ref{claim:ind} by induction on $(p,r) \in I_n$,
in \S\ref{subsec:concl} below.

Claim \ref{claim:ind} in the base case $(p,r) = (1,0)$ is an immediate consequence of resolution of singularities.  We
will need to consider item (2) of the claim with the following caveat: In the desingularization algorithm, each blowing-up
has centre given by a smooth subspace which may have several components. We allow to replace this blowing-up
by the finite number of blowings-up of the components, one at a time. Of course, the resulting morphisms are the same.
See \S\ref{subsec:MinimalHigherDim}.

Conjecture \ref{conj:min} follows from Claim \ref{claim:ind} in the
case $(p,r) = (n,0)$, for general $n$; in this case, all blowing-up are isomorphisms over the $\nc$ locus of $(X_0,E_0)$. 
Theorem \ref{thm:mainB} covers $n\leq 5$. In \S\ref{subsec:MinimalHigherDim},
we prove the claim for general $n$ and $(p,r) = (3,0)$; this gives Theorem \ref{thm:main3}.

In Theorem \ref{thm:mainB}, as well as in Theorem \ref{thm:main3}, each step of the iterative procedure
involves (A) an application of the standard desingularization algorithm, 
followed by (B) modification of the non-nc points
of a stratum $S_{p,r}$ using four additional blow-up sequences (B1)--(B4)
based on Sections \ref{sec:split}, \ref{sec:lim} and \ref{sec:triplenc} above, and
\S\S\ref{subsec:clean1}, \ref{subsec:MinimalHigherDim} below. We concentrate on Theorem \ref{thm:mainB} here,
and deal with Theorem \ref{thm:main3} in detail in \S\ref{subsec:MinimalHigherDim}.

\subsubsection{First steps}\label{subsec:first} 
Let $X$ denote an embedded hypersurface (i.e., $X\hookrightarrow Z$, where $Z$ is smooth
and $n := \dim Z = \dim X +1$), and let $E\subset Z$ denote an snc divisor.

The sequence $I$ begins $(n,0), (n-1,1), (n-1,0), (n-2,2),\ldots$.

To begin the iterative process, we use the standard desingularization algorithm to
blow up until the maximal value of $\inv$ is (at most) $\inv_{n,0}$. Then the
strata $S_{n,0}$ and $S_{n-1,1}$ are discrete, so we can blow up non-nc points of these strata to reduce to the
case that $(X,E)$ is nc in a neighbourhood of $S_{n,0} \cup S_{n-1,1}$. Set $T_{n-1,1} := S_{n,0} \cup S_{n-1,1}$,
$D_{n-1,1} := \emptyset \subset E$ and $\Sigma_{n-1,1} := T_{n-1,1} \cup D_{n-1,1}$.

We can now apply the desingularization algorithm  in the complement of $\Sigma_{n-1,1}$,
\emph{resetting the current year to year zero}, and blowing up with $\inv$-admissible centres
in the complement of $\Sigma_{n-1,1}$, stopping when the maximum value of the invariant
becomes at most $\inv_{n-1,0}$. Then the centres of blowing up involved are closed in $X$. 
Suppose the stratum $S := S_{n-1,0}$ is not empty. Then $S$ is
a smooth curve in the complement
of $\Sigma_{n-1,1}$; $S$ includes, in particular, $\nc(n-1)$ singularities and limits of 
$\nc(n-1)$ singularities of $X$.

We can blow up to eliminate any component of $S = S_{n-1,0}$ that is not generically 
normal crossings of order $n-1$.

Now, since the stratum $S$ (where $\inv_{n-1,0}$ is constant) is a smooth curve, the non-$\nc(n-1)$ points of
$S$ form a discrete subset, given by the intersection of $S$ with the exceptional divisor
(each non-$\nc(n-1)$ point is the intersection with a single component of the exceptional divisor). 

\begin{remark}\label{rem:B1}
(B1) In the general iterative step, there is an $\inv$-admissible sequence of blowings-up over the
non-nc locus in $S_{p,r}$, after which the non-nc locus
in $S_{p,r}$ lies in $E' \subset E$, where $E'$ is transverse to $S_{p,r}$
(see Lemma \ref{lemma:genericnc} and Remark \ref{rem:lim3intro}). In the case that $S$ is a curve, (B1)
is void and $E' = E$. We proceed to the following.
\end{remark}

\medskip\noindent
(B2) \emph{Splitting.} We apply Theorem \ref{thm:splitintro} in $S_{n-1,0} \cap E'$,
where we continue to write $E'$ for the appropriate transform of $E'$ above.
The blowings-up involved are again $\inv$-admissible.

\medskip\noindent
(B3) \emph{Cleaning, to get circulant normal form.} We apply Theorem \ref{thm:limintro} to obtain circulant normal
form in $S_{n-1,0} \cap E'$. The blowings-up involved are $\inv_1$-admissible. See Remark \ref{rem:reset} below.

\medskip
The preceding applies to any dimension. 

\subsubsection{Continuation in the case $n=5$}\label{subsec:next}
In this case, we proceed to modify the
non-$\nc$ points of $S_{4,0}$ to determine the neighbours of the circulant singularities given by (B3) above,
as follows.

\medskip\noindent
(B4) \emph{Moving away.}
We make a single blowing up of any non-$\nc$ point $a$ of $S_{4,0}$ to introduce a distinguished component $D_1$ of $E$,
throughout which $(X,E)$ is described by equations that are transformed from the circulant
normal form of Theorem \ref{thm:limintro} at $a$. The reason for this blowing-up is that the singularities
in a neighbourhood of a circulant point cannot be eliminated, but we do not \emph{a priori}
have good control over the limits of the neighbouring singularities arising from the previous blowings-up. 
These limiting singularities will be moved away from $D_1$.
(Compare with Remark \ref{rem:D_1trick}.)

More precisely, for each such non-$\nc$ point $a$, we make a further sequence of blowings-up with centres in $X\cap D_1$, 
following \S\ref{subsec:clean1} below, after which $\ord\, X \leq 3$ and $(X,E)$ is nc, in a neighbourhood
of $S_{4,0}\cup D_1$. The \emph{neighbours} of the given singularity in circulant normal form at $a$ are the nearby $\nc$ singularities
of $(X,E)$, together with the singularities of $(X,E)$ that live in the corresponding $D_1$.

\medskip
Let $D$ denote the union of the divisors $D_1$ above. We define $D_{4,0}$ by adjoining $D$ to $D_{4,1}$.
Let $T_{4,0}$ denote the union of (the strict transforms of) $T_{4,1}$ and $S_{4,0}$, and $\Sigma_{4,0} := T_{4,0} \cup D_{4,0}$.
These objects satisfy the following properties.
\begin{enumerate}[leftmargin=*]
\item $T_{4,0} \cap D_{4,0}$ is the set of non-nc points of $(X,E)$ in $T_{4,0}$, and $T_{4,0} \backslash D_{4,0}$
contains all $\nc$ points of order $\geq (4,0)$ of $(X_0,E_0)$.

\smallskip
\item $(X,E)$ has only minimal singularities in $D_{4,0}$ (in particular, they have smooth normalization).

\smallskip
\item There is a neighbourhood $U_{4,0}$ of $\Sigma_{4,0}$ such that $(X,E)$ is normal crossings and
$X$ has order $<4$ in $U_{4,0} \backslash \Sigma_{4,0}$.
\end{enumerate} 

\begin{remark}\label{rem:reset}
Given $(p,r) \in I$, $p\leq n-1$, let $(p,r)^+ = (p^+,r^+)$
denote the lexicographic successor of $(p,r)$; i.e.,
the smallest $(q,s) > (p,r)$ in the lexicographic order of pairs. 
When we apply the desingularization  algorithm in
the complement of $\Sigma_{p^+,r^+}$, for some $(p,r)$, as above, we first reset to year zero, and then blow up with
centres given by (A), stopping when the maximal value of $\inv$ becomes $\leq\, \inv_{p,r}$. 
In this step, the blowings-up are $\inv$-admissible for the
reset desingularization invariant (see \S\ref{subsec:inv}). It is important that 
the blowings-up involved in (B1) and (B2) are also
$\inv$-admissible because the following procedure (B3) is based on Theorems \ref{thm:limintro} and
\ref{thm:lim3intro}, which have hypotheses involving $\inv$. On the other hand, the blowings-up involved in (B3) and (B4)
are admissible (see Definition \ref{def:admiss}), but not necessarily
$\inv$-admissible. This is the reason that we have to reset to year zero in the next iterative step.

When we apply (B3) or (B4), or proceed to the next steps, we will continue
to use the notation $S_{p,r}$ for the successive strict transforms of the latter (following our convention
for the strict transforms of $X$); likewise for $T_{p,r}$ and $\Sigma_{p,r}$.
\end{remark}

We now continue to the next step. The stratum $S_{3,2}$ is discrete, so we treat it like $S_{5,0}$, $S_{4,1}$
above. Then we set $T_{3,2} := T_{4,0} \cup S_{3,2}$, $D_{3,2} := D_{4,0}$ and $\Sigma_{3,2} := T_{3,2} \cup D_{3,2}$,
and we proceed to the stratum $S_{3,1}$, repeating the process above: 

We first reset to year zero, and apply the standard
desingularization algorithm in the complement of $\Sigma_{3,2}$, stopping when the maximum value of $\inv \leq \inv_{3,1}$.
Note that all centres of blowing up are closed in $X$ (in fact, the desingularization morphism is the identity over
$U_{4,0} \backslash \Sigma_{4,0}$), because of property (3) above. The procedures (B1)--(B4) are repeated,
where Proposition \ref{prop:limnc3} and Theorem \ref{thm:lim3intro} now play the role of Theorems \ref{thm:splitintro}
and \ref{thm:limintro}, respectively, above. For details of (B4), see \S\S\ref{subsec:smoothcp3}, \ref{subsec:smoothcp3summary}.
We get $D_{3,1}$, $T_{3,1} := T_{3,2} \cup S_{3,1}$ and $\Sigma_{3,1} := T_{3,1} \cup D_{3,1}$, as before.

\medskip
A new element in the proof appears, however, when we pass from $(p,r) = (3,1)$ to $(3,0)$ (or, in general, when
we pass from $(p,r)$, where $r>1$, to $(p,r-1)$). As in property (3) above, there is a neighbourhood $U_{3,1}$ of $\Sigma_{3,1}$
such that, in $U_{3,1}\backslash \Sigma_{3,1}$, $(X,E)$ is normal crossings, but now only $\ord\, X \leq 3$. It will no longer be true,
when we reset to year zero and apply the desingularization algorithm in the complement of $\Sigma_{3,1}$, stopping when the
maximum value of $\inv = \inv_{3,0}$, that the centres of blowing up involved will be closed in $X$---they may have limit points in
$\Sigma_{3,1}$ (but not in $\Sigma_{4,0}$ or $\Sigma_{3,2}$, nor at $\nc$ points). Nevertheless, the centres of blowing up extend to admissible
centres of blowing up for $(X,E)$, and the blowings-up preserve the minimal singularities at the limit points. 
For details, see Remark \ref{rem:AdaptationExpCp3}.

\begin{remark}\label{rem:prod}
Since $X$ is a hypersurface, when $r>0$ and we apply the desingularization algorithm after resetting to year zero,
as above, the desingularization blow-up sequence for $(X,E)$ is the same as that for the ideal given by the product
of the ideal $\cI_X$ of $X$ in $\cO_Z$, and the ideals $\cI_H$ of all components $H$ of $E$. On the level of the
desingularization invariant, $\inv_{p+r,0}$ for the product ideal replaces $\inv_{p,r}$ for $(X,E)$. 
The implication for the subsequent splitting and cleaning steps (B2) and (B3) is that Theorem \ref{thm:splitintro} or
Proposition \ref{prop:limnc3} for (B2), or Theorem \ref{thm:limintro} or Proposition \ref{prop:lim3} for (B3) are simply
applied with $f$ given by a local generator of the product ideal. It is therefore not necessary to rewrite the statements
of these results to explicitly mention the case $r>0$. See Remarks \ref{rem:Dsq}, \ref{rem:limE}, \ref{rem:limnc3E} and \ref{rem:lim3E}.

The resulting circulant normal form from Theorems \ref{thm:limintro} or \ref{thm:lim3intro} for $(X,E)$ at a point of the stratum
$S_{p,r}$ will be a product of circulant singularities---the local normal form for $\cI_X$ times the $r$ smooth factors
corresponding to the components of $E$. So, if $\cI_X$ is $\cp(2)$, for example, we will write $\exc^r \times \cp(2)$ for
the local normal form of $(X,E)$.
\end{remark}

In the case $n = 5$, after dealing with the stratum $S_{3,0}$ (in a manner analogous to but simpler
than $S_{4,0}$; see \S\ref{subsec:cp4summary}), we still have to treat the strata $S_{2,r}$, $r\leq 3$, 
and $S_{1,r}$, $r\leq 4$, to complete the proof of Theorem \ref{thm:mainB}.
For Theorem \ref{thm:main3}, we will need to treat $S_{2,r}$, $r\leq n-2$, 
followed by $S_{1,r}$, $r\leq n-1$, in any dimension $n$. Details will be
provided in \S\ref{subsec:MinimalHigherDim}. This will complete the proofs of Theorems \ref{thm:mainB} and \ref{thm:main3}.

\subsection{Minimal singularities in five variables}\label{subsec:clean1}
This subsection provides details of the blow-up procedure (B4) for the strata $S_{p,r}$,
$(p,r) = (4,0)$ or $(3,1)$, in five variables, as well as for $(p,r) = (3,0)$ or $(2,0)$, in arbitrary ambient dimension $n$.
The cases $(2,r)$, $n-2 \geq r \geq 1$, are treated in \S\ref{subsec:MinimalHigherDim}.

We assume that $X$ has normal form given by Theorem \ref{thm:limintro} at a non-nc point $a$ 
of $S = S_{4,0}$; i.e., by a product of circulant 
singularities---either $\cp(4)$, $\smooth \times \cp(3)$, $\cp(2) \times \cp(2)$, or $\smooth \times \smooth \times \cp(2)$.
The case $\cp(4)$ is the most intricate, and we carry it out in detail.

\begin{remark}\label{rem:distinguishedD}
In all cases below, the distinguished divisor $D_1$ plays an important part, as explained in \S\ref{subsec:next} above.
The use of several divisors in place of $D_1$ provides more flexibility, in general; for example, we can get an analogue of 
\S\ref{subsec:cp4} following for $\cp(k)$, using $2(k-1)$ blowings-up (in particular, $6$ blowings-up for $\cp(4)$, as opposed to $9$ below). Since the use of a single distinguished divisor suffices for all results in this article, we leave the more general approach to a subsequent work.
\end{remark}

\subsubsection{Circulant point $\cp(4)$}\label{subsec:cp4}  
Let us write $\De := \De_4$.
There are \'etale coordinates $(w,x,z) = (w,x_1,x_2,x_3,z)$ at $a=0$
in which $X$ is the vanishing locus of 
$$
\De\left(z,\,w^{1/4}x_1,\,w^{2/4}x_2,\,w^{3/4}x_3\right) = \prod_{\ell=0}^3 \left(z + \ep^\ell w^{1/4}x_1 + \ep^{2\ell} w^{2/4}x_2 + \ep^{3\ell} w^{3/4}x_3\right),
$$
where $\ep = e^{2\pi i/4}$ and $\{w=0\}$ is a component of the exceptional divisor.
(We will call $\{w=0\}$ the \emph{old exceptional divisor} $D_{\text{old}}$.)
Let us enumerate the singularities of $X$ in $\{z=w=0\}$. In this 3-dimensional subspace, $X$ is smooth at a point where $x_1\neq 0$
(with tangential exceptional divisor $D_{\text{old}}$).
In $\{z=w=x_1=0\}$:
\begin{enumerate}
\item at any nonzero point of the $x_3$-axis, $X$ has order 3, and is given by the vanishing locus of
$$
\De\left(z,\,w^{1/4}x_1,\,w^{2/4}x_2,\,w^{3/4}\right),
$$
after a change of variable to absorb the unit $x_3$;
\item at any nonzero point of the $x_2$-axis, $X$ has order $2$, and is given by the vanishing locus of
$$
\De\left(z,\,w^{1/4}x_1,\,w^{2/4},\,w^{3/4}x_3\right),
$$
after a change of variable to absorb $x_2$;
\item at any point where $z=w=x_1=0$, $x_2\neq 0$, $x_3\neq 0$, $X$ also has order $2$, and is given by the vanishing locus of
$$
\De\left(z,\,w^{1/4}x_1,\,w^{2/4},\,w^{3/4}\right).
$$
\end{enumerate}
Let us explain why $X$ has isomorphic singularities at any two points in 
$\{z=w=x_1=0,\, x_2\neq 0,\, x_3\neq 0\}$; i.e., in (3) above. Note that $\De$ is homogeneous with
respect to $(x,z)$, but also weighted homogeneous with respect to $(w,x,z)$; i.e.,
$$
\De(t\cdot (w,x,z)) = t^4 \De(w,x,z),
$$
where
$$
t\cdot (w,x,z) := (tw, t^{3/4}x_1, t^{2/4}x_2, t^{1/4}x_3, tz).
$$
By homogeneity, $X$ has isomorphic singularities at any points of a curve (parametrized by $t$)
coming from either notion of homogeneity. But the families of curves coming from either notion of
homogeneity each foliate $\{z=w=x_1=0,\, x_2\neq 0,\, x_3\neq 0\}$, and any pair of curves, one
from each family, intersect.

As an essentially equivalent explanation, 
$$
\frac{1}{x_2^4}\,\De = \De\left(\frac{z}{x_2},\, w^{1/4}\,\frac{x_1}{x_2},\, w^{2/4},\, w^{3/4}\,\frac{x_3}{x_2}\right),
$$
so that
$$
\left(\frac{x_3}{x_2}\right)^8\frac{1}{x_2^4}\,\De 
= \De\left(\left(\frac{x_3}{x_2}\right)^2\frac{z}{x_2},\,\, w^{1/4}\left(\frac{x_3}{x_2}\right)^2\frac{x_1}{x_2},\,\,
w^{2/4}\left(\frac{x_3}{x_2}\right)^2,\,\, w^{3/4}\,\left(\frac{x_3}{x_2}\right)^3\right).
$$
We can now absorb units into $w,x_1,z$ to get the the normal form of item (3) above.

We will now give the remainder of the procedure (B4) of 
the minimal singularities algorithm for $\cp(4)$; i.e., we give a finite sequence of blowings-up
needed to obtain a finite collection of \emph{minimal singularities} occurring as \emph{neighbours} of
$\cp(4)$ (i.e., occurring in a small neighbourhood of $\cp(4)$ or as a limit of singularities in a neighbourhood).
The neighbours of $\cp(4)$ are the three singularities (1),\,(2),\,(3) above, together with a variant ($2'$) of (2),
all of which are listed in \S\ref{subsec:cp4summary} below.

\setlist[description]{leftmargin=\parindent}
\smallskip
\begin{description}
\item[Blow-up 1.] \emph{Introduction of a distinguished exceptional divisor $D_1$.}
Centre $= \cp(4) = 0$ in the coordinate chart above. The blowing-up is covered
by 5 coordinate charts, in each of which we will retain the same notation $(w,x_1,x_2,x_3,z)$ for the coordinates,
using the following convention.

\smallskip
\begin{description}
\item[$z$-chart.] We substitute $(wz,x_1z,x_2z,x_3z,z)$ for the original coordinates, and factor $z^4$ to
obtain the strict transform of $X$ as the vanishing locus of 
$$
\De\left(1,\,w^{1/4}z^{1/4}x_1,\,w^{2/4}z^{2/4}x_2,\,w^{3/4}z^{3/4}x_3\right).
$$
We do not need to examine this chart because the strict transform of $X$ lies entirely in the remaining charts,
following.

\smallskip
\item[$w$-chart.] We substitute $(w,wx_1,wx_2,wx_3,wz)$ to get
$$
\De\left(z,\,w^{1/4}x_1,\,w^{2/4}x_2,\,w^{3/4}x_3\right)
$$
for the strict transform. This is the same as the original formula, but the meaning of $w$ has changed---here
$\{w=0\}$ is the new exceptional divisor $D_1$ (the inverse image of the centre of blowing up), and $D_{\text{old}}$
has been moved away. Subsequent
blowings-up will have centres in $D_1$ or its successive strict transforms (which we continue to label as $D_1$).

\smallskip
\item[$x_1$-chart.] The substitution $(wx_1,x_1,x_1x_2,x_1x_3,x_1z)$ gives
$$
\De\left(z,\,w^{1/4}x_1^{1/4},\,w^{2/4}x_1^{2/4}x_2,\,w^{3/4}x_1^{3/4}x_3\right),
$$
and $D_1 = \{x_1=0\}$. (In the remaining charts, we do not write the substitution explicitly; it will follow the
same pattern, and we will describe only the strict transform and exceptional divisor. In each chart, $D_{\text{old}}$
is present as $\{w=0\}$, unless $\{w=0\}$ represents another component of $E$ as indicated, in which case $D_{\text{old}}$
does not intersect the chart.)

\medskip
\item[$x_2$-chart.] $\De\left(z,\,w^{1/4}x_2^{1/4}x_1,\,w^{2/4}x_2^{2/4},\,w^{3/4}x_2^{3/4}x_3\right)$, $D_1 = \{x_2=0\}$.

\medskip
\item[$x_3$-chart.] $\De\left(z,\,w^{1/4}x_3^{1/4}x_1,\,w^{2/4}x_3^{2/4}x_2,\,w^{3/4}x_3^{3/4}\right)$, $D_1 = \{x_3=0\}$.
\end{description}

\smallskip
\item[Blow-up 2.] Centre $=$ points of order 4 outside $\cp(4)$; this centre of blowing up is given by $D_1 \cap \{z=w=x_1=0\}$
in the $x_2$- and $x_3$-charts above. The effect of this blowing-up is to separate the $w$- and $x_3$-axes in the $x_3$-chart,
or the $w$- and $x_2$-axes in the $x_2$-chart. 

Over the $x_2$-chart, we will have 4 charts which we label as the $x_2z$-, $x_2w$-,
$x_2x_1$-, $x_2x_2$-charts, following the pattern above. We need not consider either the $x_2z$-chart (like the $z$-chart above)
or the $x_2x_2$-chart, which does not intersect (the strict transform of) $D_1$. Likewise, we do not have to consider the
$x_3z$- or $x_3x_3$-charts. Let us describe the strict transform of $X$ along with $D_1$ and the new exceptional divisor
$D_2$ in the four remaining charts.

\smallskip
\setlist[description]{leftmargin=\parindent}
\begin{description}
\item[$x_3w$-chart.] This is obtained from the substitution $(wz, w, wx_1, x_2, wx_3)$ with respect to the coordinates
of the $x_3$-chart, so we have
$$
\De\left(z,\,w^{2/4}x_3^{1/4}x_1,\,x_3^{2/4}x_2,\,w^{2/4}x_3^{3/4}\right),\
D_1 = \{x_3=0\},\ D_2 = \{w=0\}.
$$

\item[$x_3x_1$-chart.] 
\begin{equation*}
\begin{split}
\De\left(z,\,w^{1/4}x_1^{2/4}x_3^{1/4},\,w^{2/4}x_3^{2/4}x_2,\,w^{3/4}\hspace{-.4em}\right.&\left.x_1^{2/4}x_3^{3/4}\right),\\
&D_1 = \{x_3=0\},\ D_2 = \{x_1=0\}.
\end{split}
\end{equation*}

\item[$x_2w$-chart.] 
$$
\De\left(z,\,w^{2/4}x_2^{1/4}x_1,\,x_2^{2/4},\,w^{2/4}x_2^{3/4}x_3\right),\ D_1 = \{x_2=0\},
\ D_2 = \{w=0\}.
$$

\item[$x_2x_1$-chart.] 
\begin{equation*}
\begin{split}
\De\left(z,\,w^{1/4}x_1^{2/4}x_2^{1/4},\,w^{2/4}x_2^{2/4},\,w^{3/4}x_1^{2/4}\hspace{-.4em}\right.&\left.x_2^{3/4}x_3\right),\\
&D_1 = \{x_2=0\},\ D_2 = \{x_1=0\}.
\end{split}
\end{equation*}
\end{description}

\smallskip
\item[Blow-up 3.] Centre $= 0$ in the $x_3w$-chart---an isolated point of order 4. As above, we need to give the
strict transform of $X$ only in the $x_3ww$-, $x_3wx_2$- and $x_3wx_1$-charts.

\smallskip
\setlist[description]{leftmargin=\parindent}
\begin{description}
\item[$x_3ww$-chart.] 
\begin{equation*}
\begin{split}
\De\left(z,\,w^{3/4}x_3^{1/4}x_1,\,w^{2/4}x_3^{2/4}x_2,\,w^{1/4}\hspace{-.4em}\right.&\left.x_3^{3/4}\right),\\
&D_1 = \{x_3=0\},\ D_3 = \{w=0\};
\end{split}
\end{equation*}
$D_2$ has been moved away.

\smallskip
\item[$x_3wx_2$-chart.] 
\begin{equation*}
\begin{split}
\De\left(z,\,w^{2/4}x_2^{3/4}x_3^{1/4}x_1,\,\hspace{-.4em}\right.&\left.x_2^{2/4}x_3^{2/4},\,w^{2/4}x_2^{1/4}x_3^{3/4}\right),\\
&D_1 = \{x_3=0\},\ D_2 = \{w=0\},\ D_3 = \{x_2=0\}.
\end{split}
\end{equation*}

\item[$x_3wx_1$-chart.] 
\begin{equation*}
\begin{split}
\De\left(z,\,w^{2/4}x_1^{3/4}x_3^{1/4},\,\hspace{-.4em}\right.&\left.x_1^{2/4}x_3^{2/4}x_2,\,w^{2/4}x_1^{1/4}x_3^{3/4}\right),\\
&D_1 = \{x_3=0\},\ D_2 = \{w=0\},\ D_3 = \{x_1=0\}.
\end{split}
\end{equation*}
\end{description}

\smallskip
\item[Blow-up 4.] Centre $=$ points of order 4 given by $D_{\text{old}} \cap D_1 \cap D_2 \cap \{z=0\}$, in the
$x_3x_1$- and $x_2x_1$- charts. We need only consider the following:

\smallskip
\setlist[description]{leftmargin=\parindent}
\begin{description}
\item[$x_3x_1w$-chart.]
\begin{equation*}
\begin{split}
\De\left(z,\,x_1^{2/4}x_3^{1/4},\,x_3^{2/4}x_2,\,\hspace{-.4em}\right.&\left.wx_1^{2/4}x_3^{3/4}\right),\\
&D_1 = \{x_3=0\},\ D_2 = \{x_1=0\},\ D_4 = \{w=0\}.
\end{split}
\end{equation*}

\item[$x_3x_1x_1$-chart.]
\begin{equation*}
\begin{split}
\De\left(z,\,w^{1/4}x_3^{1/4},\,w^{2/4}x_3^{2/4}x_2,\,w^{3/4}\hspace{-.4em}\right.&\left.x_1x_3^{3/4}\right),\\
&D_1 = \{x_3=0\},\ D_4 = \{x_1=0\}.
\end{split}
\end{equation*}

\item[$x_2x_1w$-chart.]
\begin{equation*}
\begin{split}
\De\left(z,\,x_1^{2/4}x_2^{1/4},\,x_2^{2/4},\,w\hspace{-.4em}\right.&\left.x_1^{2/4}x_2^{3/4}x_3\right),\\
&D_1 = \{x_2=0\},\ D_2 = \{x_1=0\},\ D_4 = \{w=0\}.
\end{split}
\end{equation*}

\item[$x_2x_1x_1$-chart.]
\begin{equation*}
\begin{split}
\De\left(z,\,w^{1/4}x_2^{1/4},\,w^{2/4}x_2^{2/4},\,w^{3/4}x_1\hspace{-.4em}\right.&\left.x_2^{3/4}x_3\right),\\
&D_1 = \{x_2=0\},\ D_4 = \{x_1=0\}.
\end{split}
\end{equation*}
\end{description}

\smallskip
\item[Blow-up 5.] Centre $=$ points of order 4 given by $D_1 \cap D_3 \cap \{z=0\}$, appearing in the
three charts of blow-up 3; i.e., in the  $x_3ww$-, $x_3wx_2$- and $x_3wx_1$-charts. We need only consider
a single chart in each case.

\smallskip
\setlist[description]{leftmargin=\parindent}
\begin{description}
\item[$x_3www$-chart.] 
\begin{equation*}
\De\left(z,\,x_3^{1/4}x_1,\,x_3^{2/4}x_2,\,x_3^{3/4}\right),\
D_1 = \{x_3=0\},\ D_5 = \{w=0\}.
\end{equation*}
This singularity is a neighbour (1) of $\cp(4)$.

\smallskip
\item[$x_3wx_2x_2$-chart.] 
\begin{equation*}
\begin{split}
\De\left(z,\,w^{2/4}x_3^{1/4}x_1,\,\hspace{-.4em}\right.&\left.x_3^{2/4},\,w^{2/4}x_3^{3/4}\right),\\
&D_1 = \{x_3=0\},\ D_2 = \{w=0\},\ D_5 = \{x_2=0\}.
\end{split}
\end{equation*}

\item[$x_3wx_1x_1$-chart.] 
\begin{equation*}
\begin{split}
\De\left(z,\,w^{2/4}x_3^{1/4},\,\hspace{-.4em}\right.&\left.x_3^{2/4}x_2,\,w^{2/4}x_3^{3/4}\right),\\
&D_1 = \{x_3=0\},\ D_2 = \{w=0\},\ D_5 = \{x_1=0\}.
\end{split}
\end{equation*}
\end{description}

\smallskip
\item[Blow-up 6.] Centre $= D_1 \cap D_2 \cap \{z=0\}$, generically of order 2, appearing in the
$x_2w$-, $x_3x_1w$-, $x_2x_1w$-, $x_3wx_2x_2$- and $x_3wx_1x_1$-charts. We need to consider
the following.

\smallskip
\setlist[description]{leftmargin=\parindent}
\begin{description}
\item[$x_2wz$-chart.] 
\begin{equation*}
\begin{split}
\De\left(z^{2/4},\,z^{1/4}w^{2/4}x_2^{1/4}x_1,\, \hspace{-.4em}\right.&\left.x_2^{2/4},\, z^{3/4}w^{2/4}x_2^{3/4}x_3\right),\\
&D_1 = \{x_2=0\},\, D_2 = \{w=0\},\, D_6 = \{z=0\}.
\end{split}
\end{equation*}

\item[$x_2ww$-chart.] 
$$
\De\left(w^{2/4}z,\,w^{1/4}x_2^{1/4}x_1,\,x_2^{2/4},\,w^{3/4}x_2^{3/4}x_3\right),\ D_1 = \{x_2=0\},
\ D_6 = \{w=0\}.
$$

\item[$x_3x_1wz$-chart.]
\begin{equation*}
\begin{split}
\De\left(z^{2/4},\, z^{1/4}\hspace{-.4em}\right.&\left.x_1^{2/4}x_3^{1/4},\, x_3^{2/4}x_2,\, wz^{3/4}x_1^{2/4}x_3^{3/4}\right),\\
&D_1 = \{x_3=0\},\ D_2 = \{x_1=0\},\ D_4 = \{w=0\},\  D_6 = \{z=0\}.
\end{split}
\end{equation*}

\item[$x_3x_1wx_1$-chart.]
\begin{equation*}
\begin{split}
\De\left(x_1^{2/4}z,\,x_1^{1/4}x_3^{1/4},\,x_3^{2/4}x_2,\,\hspace{-.4em}\right.&\left.wx_1^{3/4}x_3^{3/4}\right),\\
&D_1 = \{x_3=0\},\ D_4 = \{w=0\},\ D_6 = \{x_1=0\}.
\end{split}
\end{equation*}

\item[$x_2x_1wz$-chart.]
\begin{equation*}
\begin{split}
\De\left(z^{2/4},\,z^{1/4}x_1^{2/4}x_2^{1/4},\,x_2^{2/4},\, \hspace{-.4em}\right.&\left.wz^{3/4}x_1^{2/4}x_2^{3/4}x_3\right),\\
&D_1 = \{x_2=0\},\ D_4 = \{w=0\},\ D_6 = \{z=0\}.
\end{split}
\end{equation*}

\item[$x_2x_1wx_1$-chart.]
\begin{equation*}
\begin{split}
\De\left(x_1^{2/4}z,\,x_1^{1/4}x_2^{1/4},\,x_2^{2/4},\,w\hspace{-.4em}\right.&\left.x_1^{3/4}x_2^{3/4}x_3\right),\\
&D_1 = \{x_2=0\},\ D_4 = \{w=0\},\ D_6 = \{x_1=0\}.
\end{split}
\end{equation*}

\item[$x_3wx_2x_2w$-chart.] 
\begin{equation*}
\begin{split}
\De\left(w^{2/4}z,\,w^{1/4}x_3^{1/4}x_1,\,\hspace{-.4em}\right.&\left.x_3^{2/4},\,w^{3/4}x_3^{3/4}\right),\\
&D_1 = \{x_3=0\},\ D_5 = \{x_2=0\},\ D_6 = \{w=0\}.
\end{split}
\end{equation*}

\item[$x_3wx_1x_1w$-chart.] 
\begin{equation*}
\begin{split}
\De\left(w^{2/4}z,\,w^{1/4}x_3^{1/4},\,\hspace{-.4em}\right.&\left.x_3^{2/4}x_2,\,w^{3/4}x_3^{3/4}\right),\\
&D_1 = \{x_3=0\},\ D_5 = \{x_1=0\},\ D_6 = \{w=0\}.
\end{split}
\end{equation*}
\end{description}

\smallskip
\item[Blow-up 7.] Centre $= D_1 \cap D_6$. We need to consider only the following charts, over each of
the preceding.

\smallskip
\setlist[description]{leftmargin=\parindent}
\begin{description}
\item[$x_2www$-chart.] 
$$
\De\left(z,\,x_2^{1/4}x_1,\,x_2^{2/4},\,wx_2^{3/4}x_3\right),\ D_1 = \{x_2=0\},
\ D_7 = \{w=0\}.
$$
This is a neighbour of $\cp(4)$ (with smooth normalization); see ($2'$) in \S\ref{subsec:cp4summary} below.

\smallskip
\item[$x_3x_1wx_1x_1$-chart.]
\begin{equation*}
\begin{split}
\De\left(z,\,x_3^{1/4},\,x_3^{2/4}x_2,\,\hspace{-.4em}\right.&\left.wx_1x_3^{3/4}\right),\\
&D_1 = \{x_3=0\},\ D_4 = \{w=0\},\ D_7 = \{x_1=0\}.
\end{split}
\end{equation*}
This is smooth.

\smallskip
\item[$x_2x_1wx_1x_1$-chart.]
\begin{equation*}
\begin{split}
\De\left(z,\,x_2^{1/4},\,x_2^{2/4},\,w\hspace{-.4em}\right.&\left.x_1x_2^{3/4}x_3\right),\\
&D_1 = \{x_2=0\},\ D_4 = \{w=0\},\ D_7 = \{x_1=0\}.
\end{split}
\end{equation*}
Smooth again.

\smallskip
\item[$x_3wx_2x_2ww$-chart.] 
\begin{equation*}
\begin{split}
\De\left(z,\,x_3^{1/4}x_1,\,x_3^{2/4},\,\hspace{-.4em}\right.&\left.wx_3^{3/4}\right),\\
&D_1 = \{x_3=0\},\ D_5 = \{x_2=0\},\ D_7 = \{w=0\}.
\end{split}
\end{equation*}
A neighbour (2) of cp4.

\smallskip
\item[$x_3wx_1x_1ww$-chart.] 
\begin{equation*}
\begin{split}
\De\left(z,\,x_3^{1/4},\,x_3^{2/4}x_2,\,\hspace{-.4em}\right.&\left.wx_3^{3/4}\right),\\
&D_1 = \{x_3=0\},\ D_5 = \{x_1=0\},\ D_7 = \{w=0\}.
\end{split}
\end{equation*}
Smooth again.
\end{description}

\smallskip
\item[Blow-up 8.] Centre $= D_{\text{old}} \cap D_1 \cap \{z=0\}$, appearing in the $x_1$-,
$x_3x_1x_1$- and $x_2x_1x_1$-charts. Over these three charts, we need to consider
the following.

\smallskip
\setlist[description]{leftmargin=\parindent}
\begin{description}
\item[$x_1z$-chart.] 
\begin{equation*}
\begin{split}
\De\left(z^{2/4},\,w^{1/4}x_1^{1/4},\, z^{2/4}w^{2/4}x_1^{2/4}x_2,\, \hspace{-.4em}\right.&\left.zw^{3/4}x_1^{3/4}x_3\right),\\
&D_1 = \{x_1=0\},\ D_8 = \{z=0\}.
\end{split}
\end{equation*}

\item[$x_1w$-chart.] 
$$
\De\left(w^{2/4}z,\,x_1^{1/4},\,w^{2/4}x_1^{2/4}x_2,\,wx_1^{3/4}x_3\right),\ D_1 = \{x_1=0\},\ D_8 = \{w=0\}.
$$
This is smooth.

\smallskip
\item[$x_3x_1x_1z$-chart.]
\begin{equation*}
\begin{split}
\De\left(z^{2/4},\, w^{1/4}x_3^{1/4},\, z^{2/4}\hspace{-.4em}\right.&\left.w^{2/4}x_3^{2/4}x_2,\, zw^{3/4}x_1x_3^{3/4}\right),\\
&D_1 = \{x_3=0\},\ D_4 = \{x_1=0\},\ D_8 = \{z=0\}.
\end{split}
\end{equation*}

\item[$x_3x_1x_1w$-chart.]
\begin{equation*}
\begin{split}
\De\left(w^{2/4}z,\,x_3^{1/4},\,w^{2/4}x_3^{2/4}\hspace{-.4em}\right.&\left.x_2,\,wx_1x_3^{3/4}\right),\\
&D_1 = \{x_3=0\},\ D_4 = \{x_1=0\},\ D_8 = \{w=0\}.
\end{split}
\end{equation*}
Smooth again.

\smallskip
\item[$x_2x_1x_1z$-chart.]
\begin{equation*}
\begin{split}
\De\left(z^{2/4},\, w^{1/4}x_2^{1/4},\, z^{2/4}\hspace{-.4em}\right.&\left.w^{2/4}x_2^{2/4},\, zw^{3/4}x_1x_2^{3/4}x_3\right),\\
&D_1 = \{x_2=0\},\ D_4 = \{x_1=0\},\ D_8 = \{z=0\}.
\end{split}
\end{equation*}

\item[$x_2x_1x_1w$-chart.]
\begin{equation*}
\begin{split}
\De\left(w^{2/4}z,\,x_2^{1/4},\,w^{2/4}x_2^{2/4},\,\hspace{-.4em}\right.&\left.wx_1x_2^{3/4}x_3\right),\\
&D_1 = \{x_2=0\},\ D_4 = \{x_1=0\},\ D_8 = \{w=0\}.
\end{split}
\end{equation*}
Smooth again.
\end{description}

\smallskip
\item[Blow-up 9.] Centre $= D_{\text{old}} \cap D_1 \cap D_8$, appearing in the $x_1z$-,
$x_3x_1x_1z$- and $x_2x_1x_1z$-charts. Over these three charts, we need only consider
the following, and they are all smooth.

\smallskip
\setlist[description]{leftmargin=\parindent}
\begin{description}
\item[$x_1zw$-chart.] 
\begin{equation*}
\begin{split}
\De\left(z^{2/4},\,x_1^{1/4},\, z^{2/4}wx_1^{2/4}\hspace{-.4em}\right.&\left.x_2,\, zw^2x_1^{3/4}x_3\right),\\
&D_1 = \{x_1=0\},\ D_8 = \{z=0\},\ D_9 = \{w=0\}.
\end{split}
\end{equation*}

\item[$x_3x_1x_1zw$-chart.]
\begin{equation*}
\begin{split}
\De\left(z^{2/4},\, \hspace{-.4em}\right.&\left.x_3^{1/4},\, z^{2/4}wx_3^{2/4}x_2,\, zw^2x_1x_3^{3/4}\right),\\
&D_1 = \{x_3=0\},\ D_4 = \{x_1=0\},\ D_8 = \{z=0\},\ D_9 = \{w=0\}.
\end{split}
\end{equation*}

\item[$x_2x_1x_1zw$-chart.]
\begin{equation*}
\begin{split}
\De\left(z^{2/4},\, \hspace{-.4em}\right.&\left.x_2^{1/4},\, z^{2/4}wx_2^{2/4},\, zw^2x_1x_2^{3/4}x_3\right),\\
&D_1 = \{x_2=0\},\ D_4 = \{x_1=0\},\ D_8 = \{z=0\},\ D_9 = \{w=0\}.
\end{split}
\end{equation*}
\end{description}

\end{description}

\smallskip
There is of course some flexibility in the choice of the preceding blowings-up; for example, 
blow-ups 8,\,9 could have been performed before blow-ups 6,\,7, and we may switch the order of 3 and 4,
or of 4 and 5.

\subsubsection{Summary of the $\cp(4)$ case}\label{subsec:cp4summary} 
After the preceding sequence of blowings-up, only singularities $\{\De_4=0\}$ of the
following kind appear in the exceptional divisor $D_1$. (Here we have re-labelled coordinates to be
consistent with the normal forms (1)-(3) above.)
\begin{enumerate}
\item\ $\De_4\left(z,\,w^{1/4}x_1,\,w^{2/4}x_2,\,w^{3/4}\right)$,
\smallskip
\item\ $\De_4\left(z,\,w^{1/4}x_1,\,w^{2/4},\,w^{3/4}x_3\right)$,
\smallskip
\item[($2'$)]\ $\De_4\left(z,\,w^{1/4}x_1,\,w^{2/4},\,w^{3/4}x_2x_3\right)$,
\smallskip
\item[(3)]\ $\De_4\left(z,\,w^{1/4}x_1,\,w^{2/4},\,w^{3/4}\right)$.
\end{enumerate}
\smallskip
These singularities are the \emph{neighbours} of $\cp(4)$.
In (2), $x_3$ may or may not represent an exceptional divisor, and in ($2'$), $x_2$ represents an exceptional
divisor. In (1) or (2), moreover, there may be an additional exceptional divisor $x_3$ or $x_2$ (respectively).

Note that the nearby singularities outside $D_1$ are only normal crossings singularities because, in the order 2
cases (2), ($2'$), (3) (respectively, in the order 3 case (1)), the gradients of any two factors (respectively, of any
three factors) of $\De=\De_4$ are linearly independent at such a nearby point. Moreover, $X$ and the exceptional divisor $E$
are simultaneously normal crossings at nearby points outside $D_1$. 

We summarize these results in the following lemma (where, as usual, we use the same notation $(X,E)$, etc.,
for the transforms of our objects after a sequence of blowings-up).

\begin{lemma}\label{lem:cp4summary} 
After first blowing up a $\cp(4)$ point to introduce a new exceptional divisor $D_1$, there is a sequence
of seven admissible blowings-up with centres in $D_1$, after which
\begin{enumerate}
\item $X$ has only minimal non-$\nc$ singularities as above (besides the $\cp(4)$ point), and therefore smooth normalization, at
points of $D_1$;
\item there is a neighbourhood $U$ of $D_1$ such that $(X,E)$ has only $\nc$ singularities in $U\backslash D_1$, which are of order $< (4,0)$ outside $S_{4,0}$.
\end{enumerate}
\end{lemma}

\smallskip
\begin{remark}\label{rem:cp3,2} \emph{Circulant singularities of lower order.} The cases
\begin{align*} 
\cp(3)  &\qquad \De_3\left(z,\,w^{1/3}x_1,\,w^{2/3}x_2\right),\\
\cp(2)  &\qquad \De_2\left(z,\,w^{1/2}x\right) \quad \text{(pinch point)}
\end{align*}
are much
simpler versions of the $\cp(4)$ case above (see \cite{BLMmin2}). In particular, $\cp(3)$ has only one singular neighbour
\begin{equation}\label{eq:dpp}
\De_3\left(z,\,w^{1/3}x_1,\,w^{2/3}\right)
\end{equation}
in the exceptional divisor $D_1 = \{w=0\}$ (this singularity was called a \emph{degenerate pinch point}
in \cite{BLMmin2}), and $\cp(2)$ has only a smooth neighbour in $D_1$. After the first blowing-up to
introduce $D_1$, only three additional blowings-up are needed for $\cp(3)$, and only one for $\cp(2)$.

Moreover, following Theorem \ref{thm:lim3intro} and Proposition \ref{prop:lim3} in the case of an irreducible limit
of $\nc(3)$, in any dimension (and the simpler version for limits of $\nc(2)$; cf.\,\cite{BMmin1}), we get the preceding normal forms
of $\cp(3)$ and $\cp(2)$ (independent of the remaining variables), and we obtain the neighbours
above, by global blowings-up. See also Remark \ref{rem:etcsummary} below.
\end{remark}

In five variables, apart from $\cp(4)$, singularities of the following three kinds may occur at an isolated point
of the stratum $S_{4,0}$.

\subsubsection{$\mathrm{Smooth} \times \cp(3)$}\label{subsec:smoothcp3} 
Let us now write $\De = \De_3$.
There are \'etale coordinates $(w, x_1, x_2, y, z)$
in which $X$ is the vanishing locus of
$$
y\De\left(z,\,w^{1/3}x_1,\,w^{2/3}x_2\right).
$$

\setlist[description]{leftmargin=\parindent}
\smallskip
\begin{description}
\item[Blow-up 1.] \emph{Introduction of $D_1$.}
Centre $= 0$ in the coordinate chart above. The blowing-up is covered
by 5 coordinate charts, with the following transforms of the ideal of $X$ and exceptional divisor $D_1$.

\smallskip
\begin{description}
\item[$z$-chart.]\ $y$\ (smooth),\ $D_1 = \{z=0\}$.

\smallskip
\item[$w$-chart.]\  $y\De\left(z,\,w^{1/3}x_1,\,w^{2/3}x_2\right),\ D_1 = \{w=0\}$.

\smallskip
\item[$x_1$-chart.]\  $y\De\left(z,\,w^{1/3}x_1^{1/3},\,w^{2/3}x_1^{2/3}x_2\right),\ D_1 = \{x_1=0\}$.

\smallskip
\item[$x_2$-chart.]\  $y\De\left(z,\,w^{1/3}x_2^{1/3}x_1,\,w^{2/3}x_2^{2/3}\right),\ D_1 = \{x_2=0\}$.

\smallskip
\item[$y$-chart.]\  $\De\left(z,\,w^{1/3}y^{1/3}x_1,\,w^{2/3}y^{2/3}x_2\right),\ D_1 = \{y=0\}$.
\end{description}
\end{description}

\medskip
We now make three further blowings-up, which are essentially the three blowings-up needed for $\cp(3)$
after the introduction of $D_1$ (see Remark \ref{rem:cp3,2}). Over the $w$-, $x_1$- and $x_2$-charts, in fact,
these are simply the blowings-up for $\cp(3)$ in the presence of the additional variable $y$; after each blowing-up,
we get $y\, \times$ the transform of the blowing-up for $\cp(3)$. So we leave the computation
to the reader, and describe only the transforms over the $y$-chart above (where the centre of blowing up extends,
in any case, to the centre needed over the $w$-, $x_1$- and $x_2$-charts).

\setlist[description]{leftmargin=\parindent}
\smallskip
\begin{description}
\item[Blow-up 2.] Centre $=$ points of order $3$, $D_{\mathrm{old}} \cap D_1 \cap \{z = x_1 = 0\}$ in the
$y$-chart (and the $x_2$-chart). Over the $y$-chart, we need
consider only the following.

\smallskip
\begin{description}
\item[$yx_1$-chart.]
$$
\De\left(z,\,w^{1/3}y^{1/3}x_1^{2/3},\,w^{2/3}y^{2/3}x_1^{1/3}x_2\right),\ D_1 = \{y=0\},\, D_2=\{x_1=0\}.
$$

\item[$yw$-chart.]
$$
\De\left(z,\,w^{2/3}y^{1/3}x_1,\,w^{1/3}y^{2/3}x_2\right),\ D_1 = \{y=0\}\, D_2=\{w=0\}.
$$
\end{description}

\smallskip
\item[Blow-up 3.] Centre $=$ points of order $3$, $D_1 \cap D_2 \cap \{z=0\}$. We need consider only the following.

\smallskip
\begin{description}
\item[$yx_1x_1$-chart.]
$$
\De\left(z,\,w^{1/3}y^{1/3},\,w^{2/3}y^{2/3}x_2\right),\ D_1 = \{y=0\},\, D_3=\{x_1=0\}.
$$

\item[$yww$-chart.]
$$
\De\left(z,\,y^{1/3}x_1,\,y^{2/3}x_2\right),\ D_1 = \{y=0\},\, D_3=\{w=0\}.
$$
This is $\cp(3)$. 
\end{description}

\smallskip
\item[Blow-up 4.] Centre $=$ order $2$ points,
$D_{\text{old}} \cap D_1 \cap \{z=0\}$. We have to consider only
the $yx_1x_1w$-chart, where $X$ becomes smooth.
\end{description}

\subsubsection{Summary of the $\smooth \times \cp(3)$ case}\label{subsec:smoothcp3summary} 
We get the following as non-$\nc$ singular neighbours of $\smooth \times \cp(3)$, in $D_1$.
\begin{enumerate}
\item\ $\cp(3)$,
\item\ $y\De_3\left(z,\,w^{1/3}x_1,\,w^{2/3}\right)$,
\item\ $\De_3\left(z,\,w^{1/3}x_1,\,w^{2/3}\right)$.
\end{enumerate}
\smallskip
These occur already in a small neighbourhood of $\smooth \times \cp(3)$. Moreover, after the
four blowings-up above, there is a neighbourhood $U$ of $D_1$ such that $(X,E)$ has only normal
crossings singularities of order $\leq (3,1)$ in $U\backslash \left(D_1 \cup S_{4,0}\right)$. 
It follows that the desingularization invariant, where we reset the current year to year zero,
is $\leq \inv_{3,1}$ in $U\backslash \left(D_1 \cup S_{4,0}\right)$.
We will formulate
a summary lemma analogous to Lemma \ref{lem:cp4summary} covering all three cases
$\smooth \times \cp(3)$, $\cp(2) \times \cp(2)$ and $\smooth \times \smooth \times \cp(2)$; see
Lemma \ref{lem:etcsummary}.

\begin{remark}\label{rem:AdaptationExpCp3}
We treat the case $\exp \times \cp(3)$, which appears when dealing with the stratum $S_{3,1}$, using the same blow-up
sequence as for $\smooth \times \cp(3)$, where $\exp =\{y=0\}$ (see Remark \ref{rem:prod}). 
After the four blowings-up above, there is a neighbourhood $U$ of $D_1$ such that $(X,E)$ has only normal crossings singularities of order $\leq (3,1)$ in $U\backslash \left(D_1 \cup S_{3,1}\right) = U\backslash \Sigma_{3,1}$.

In order to continue to the stratum $S_{3,0}$, as in \S\ref{subsec:overviewproof}, we have to consider the desingularization morphism $\s$ over the complement
of $\Sigma_{3,1}$, where we first reset to year zero, and stop when $\inv \leq \inv_{3,0}$. Note that the centres of blowing-up involved
in $\s$ lie over only the $yww$-chart of \S\ref{subsec:smoothcp3}; in fact, $\s$ consists of a single blowing-up with centre given
in the $yww$-chart by the smooth curve $C = \{z=x_1=x_2= w=0\} \subset D_3$; 
this curve does not intersect the strict transform of $S_{3,1}$, which lies in the $w$-chart,
but it does intersect $D_1$. After blowing up with centre $C$, we already have $\inv \leq \inv_{3,0}$, and the class
of minimal singularities in $\Sigma_{3,1}$ is preserved.
\end{remark}

\subsubsection{$\cp(2) \times \cp(2)$}\label{subsec:cp2cp2} $X$ is given by
\begin{equation}\label{eq:cp2cp2}
(z_1^2 - wx_1^2)(z_2^2-wx_2^2) = 0.
\end{equation}
The non-$\nc$ singularities in a small neighbourhood of the origin are the following.
\begin{enumerate}
\item\ $(z_1^2 - w)(z_2^2-wx_2^2)$,
\item\ $(z_1^2 - w)(z_2^2-w)$,
\item\ $\cp(2)$.
\end{enumerate}

\smallskip
We again blow up the origin to introduce $D_1$. Then we get the same equation \eqref{eq:cp2cp2} in the $w$-chart.
We get the following in the $x_1$-chart:
$$
(z_1^2 - wx_1)(z_2^2-wx_1x_2^2) = 0,\ D_{\text{old}} = \{w=0\},\, D_1 = \{x_1=0\},
$$
and a symmetric description in the $x_2$ chart. Also, in the $z_1$-chart, we get
$$
z_2^2 - wz_1x_2^2 = 0,\ D_{\text{old}} = \{w=0\},\, D_1 = \{z_1=0\},
$$
and we get a symmetric description in the $z_2$-chart. 
There is a further sequence of blowings-up with centres in $D_1$, after which
we have only the preceding non-$\nc$ singularities in $D_1$, and only $\nc$ singularities of order
$\leq (3,1)$ in $U\backslash  \left(D_1 \cup S_{4,0}\right)$, where $U$ is a neighbourhood of $D_1$.
We leave the full blowing-up computation to the reader.

\subsubsection{$\mathrm{Smooth} \times \smooth \times \cp(2)$}\label{subsec:smsmcp2} The non-$\nc$ singular
neighbours are $\smooth \times \cp(2)$ and $\cp(2)$, and we get a statement similar to that in \S\ref{subsec:cp2cp2}
(cf. \S\ref{subsec:smoothcp3} above, as well as \cite{BLMmin2}).

\subsubsection{Summary lemma for the stratum $S_{4,0}$}\label{subsec:summary}

\begin{lemma}\label{lem:etcsummary}
In each case $\smooth\times\cp(3)$, $\cp(2)\times\cp(2)$ or $\smooth\times\smooth\times\cp(2)$, after first blowing up
the point to introduce an exceptional divisor $D_1$, there is a sequence of admissible blowings-up with centres in $D_1$,
after which
\begin{enumerate}
\item $X$ has only minimal non-$\nc$ singularities as listed respectively in \S\ref{subsec:smoothcp3summary},
\ref{subsec:cp2cp2} or \ref{subsec:smsmcp2}, and therefore smooth normalization, at points of $D_1$;
\item there is a neighbourhood $U$ of $D_1$ such that $(X,E)$ has only $\nc$ singularities in $U\backslash D_1$,
which are of order $< (4,0)$ outside $S_{4,0}$.
\end{enumerate}
\end{lemma}

\begin{remark}\label{rem:etcsummary}
In the simpler cases $\cp(3)$ and $\smooth\times\cp(2)$ analogous to $\cp(4)$ and $\smooth\times\cp(3)$,
respectively, the analogues of Lemmas \ref{lem:cp4summary} and \ref{lem:etcsummary} hold in any
dimension; see Remark \ref{rem:cp3,2}.
\end{remark}

\subsection{Minimal singularities of order at most $3$.}\label{subsec:MinimalHigherDim} 
This section provides details of the blow-up procedure (B4) in the remaining cases, not already covered
in \S\ref{subsec:clean1}; i.e., for the strata $S_{2,r}$ and $S_{1,r}$. We need to treat the strata $S_{3,0},\,S_{2,r}$ and $S_{1,r}$ in any
number of variables, in order that the results complete the proof not only of Theorem \ref{thm:mainB}, but also of
the following more precise version of Theorem \ref{thm:nc(3)}, for a pair $(X,E)$.

\begin{theorem}\label{thm:main3}
Given $(X,E)$ in arbitrary dimension, there is a finite sequence of admissible blowings-up \eqref{eq:blup}
such that every $\s_j$ is an isomorphism over the locus of normal crossings points of $(X_0,E_0) = (X,E)$ of
order at most $(3,0)$, and $(X_t, E_t)$
has only minimal singularities.
\end{theorem}

For the stratum $S_{3,0}$, (B4) has been covered in Remarks \ref{rem:cp3,2}, \ref{rem:etcsummary}.
Singularities $\exp\times\cp(2)$ in the stratum $S_{2,1}$ are analogous to $\smooth\times\cp(2)$ in $S_{3,0}$
(cf. Remark \ref{rem:AdaptationExpCp3}). We will carry out the details of $\exp^r\times\cp(2)$, $0<r\leq n-2$, 
for general $n$.

Consider an $\exp^r\times\cp(2)$ point $a$ in the stratum $S_{2,r}$. In this case, the exceptional divisor at $a$
can be separated into two parts: $\Eold$ corresponding to $\exp^r$, and $\Enew$ given by the components of $E\backslash \Eold$
at $a$ (introduced in the previous modifications of $S_{2,r}$).

There are \'{e}tale coordinates
\[
(w,v,y,u,x,z)= (w,v_1,\ldots,v_s,y_1,\ldots,y_r,u_1,\ldots,u_t,x,z)
\]
at $a=0$, in which
\[
X = \{z^2 - wx^2 =0\}, \quad \Eold = \{y_1\cdots y_r=0\}, \quad \Enew = \{wv_1\cdots v_s=0\}
\]
In these coordinates, $S_{2,r} = \{z=x=y=0\}$. Let us write $D_w := \{w=0\}$.

\setlist[description]{leftmargin=\parindent}
\smallskip
\begin{description}
\item[Blow-up 1.] Centre $S_{2,r} \cap D_w$, to introduce $D_1$. The blowing-up is covered by $r+3$ charts, including $r$
symmetric $y_j$-charts, so we can consider only the following.

\smallskip
\begin{description}
\item[$z$-chart.] $X \cap D_1 = \emptyset$.

\smallskip
\item[$x$-chart.] $X=\{z^2 - wx=0\},\quad \ D_1 = \{x=0\}$.

\smallskip
\item[$w$-chart.]
$$
z^2 - wx^2, \quad \Eold = \{y_1\cdots y_r=0\}, \quad \Enew = \{wv_1\cdots v_s=0\},
$$
where $D_1 =\{w=0\}$. Here, $\{z=x=y=0\}$ is (the strict transform of) $S_{2,r}$.

\smallskip
\item[$y_1$-chart.]
$$
z^2 - y_1wx^2,\quad \Eold = \{y_2\cdots y_r=0\}, \quad \Enew = \{wv_1\cdots v_sy_1=0\},
$$
where $D_1 =\{y_1=0\}$.
\end{description}

\smallskip

\item[Blow-up 2.] Centre $= \{\text{order 2 points of $X$}\} \cap D_1 \cap D_w$.

\smallskip
\begin{description}
\item[$xw$-chart.] $X$ is smooth.

\smallskip
\item[$y_1w$-chart.]
$$
z^2 - y_1x^2,\quad \Eold = \{y_2\cdots y_r=0\}, \quad \Enew = \{wv_1\cdots v_sy_1=0\},
$$
where $D_1 =\{y_1=0\}$. Note that $(X,E)$ has only $nc$-singularities outside $D_1$.
\end{description}
\end{description}

\smallskip
We now reset to year zero, and use the desingularization algorithm to blow up outside $\Sigma_{2,r}$ (i.e., outside $S_{2,r} \cup D_1$
in the coordinate charts above), stopping when the maximum value of $\inv$ is $\inv_{2,r-1}$. There is a neighbourhood
$U_{2,r}$ of $\Sigma_{2,r}$ such that $(X,E)$ is nc in $U_{2,r}\backslash D_{2,r}$ (see \S\ref{subsec:overviewproof}),
so the resolution process is essentially combinatorial over $U_{2,r}\backslash \Sigma_{2,r}$. 

Each centre of blowing up may have limits at $\cp(2)$ points of $X$ in $D_{2,r}$, and the centre of blowing up may have several disjoint
components with the same limit point. In suitable local coordinates at such a point, $X$ is given by $\{z^2 -yx^2 = 0\}$, where $\{y=0\}$
is a component of $D_{2,r}\backslash D_{3,0}$, and each component of the centre of blowing up is given by the intersection of
$\{z=x=0\}$ and at least $r$ components of $E$; therefore, at least one of these components belongs to $\Enew$, so the blowings-up
do not modify the nc locus of $(X_0,E_0)$. Clearly, each of the components extends to a closed smooth subspace of $X$, as required,
and we can blow up one at a time. Since none of the components of $E$ defining the centre of blowing-up is $\{y=0\}$, the
limiting $\cp(2)$ singularity is preserved.

For example, in the $y_1w$-chart above, the first centre of blowing up is $\{z=x=y_2=\cdots y_r=v=w=0\}$.

The blow-up sequence leads to $\inv\leq \inv_{2,r-1}$ in the complement of $\Sigma_{2,r}$, to complete the step.

Methods similar to the preceding are used in \cite[Proofs of Theorems 3.4, 1.18]{BMmin1}.

Finally, we can handle the stratum $S_{2,0}$ as in Remark \ref{rem:cp3,2}, and then deal with the strata
$S_{1,r}$ in a similar way to $S_{2,r}$, to complete the proofs of Theorems
\ref{thm:mainB} and \ref{thm:main3}. The $S_{1,r}$ case has much in common with problem of
partial desingularization preserving $\snc$, discussed in Section \ref{sec:intro}.

\subsection{Concluding remark}\label{subsec:concl}
As remarked in \S\ref{subsec:overviewproof}, Conjecture \ref{conj:min} follows from Claim \ref{claim:ind}. We propose to prove the claim by induction on $(p,r) \in I_n$, 
in the following way (see also \S\ref{subsec:approach}). Given $(p,r)$, we first apply (A) the standard 
desingularization algorithm to blow up until $\inv\leq\inv_{p,r}$. Secondly, we modify the stratum $S_{p,r}$ using
the four blowing up procedures (B1)--(B4), generalized to arbitrary dimension. 
And finally, we apply the inductive assumption for the predecessor of $(p,r)$ in $I_n$,
to the complement of $\Sigma_{p,r} := S_{p,r} \cup D_{p,r}$, where
$D_{p,r}$ denotes the union of the special divisors $D_1$ introduced in (B4) (and where the centres of blowing up
in the complement are extended as in \S\ref{subsec:MinimalHigherDim} in the case that $r>0$; 
see also Remark \ref{rem:distinguishedD}).

In view of Remark \ref{rem:forward}, the main challenge remaining for partial desingularization
preserving normal crossings in arbitrary dimension is a generalization of the moving away techniques (B4)
of this section.

\bibliographystyle{amsplain}

\end{document}